\documentclass[12pt, a4paper]{article}
\usepackage{amsfonts,amssymb,amsmath,amscd,latexsym,makeidx,theorem,graphics}
\usepackage{hyperref}
\usepackage{color}
\usepackage{mathrsfs}
\usepackage{authblk}

\usepackage[pdftex]{graphicx}

\title{Singular solutions for the constant $Q-$curvature problem}
\author{Ali Hyder\thanks{The author is supported by SNSF Grant No. P400P2-183866. } \, and Yannick Sire}
\date{}

\newtheorem{thm}{Theorem}[section]

\newtheorem{lem}[thm]{Lemma}
\newtheorem{prop}[thm]{Proposition}

\newtheorem{rem}{Remark}%[section]

\newenvironment{proof}{\noindent\emph{Proof.}}{\hfill$\square$\medskip}

\newcommand{\p}{\partial}
\newcommand{\N}{\mathbb{N}}
\newcommand{\D}{\Delta}

\newcommand{\vp}{\varphi}
\newcommand{\R}{\mathbb{R}}
\newcommand{\ve}{\varepsilon}
\newcommand{\bve}{{\bar\varepsilon}}
\newcommand{\LL}{\mathcal{L}}
\newcommand{\GG}{\mathcal{G}} 
\newcommand{\HH}{\mathbb{H}}

\begin{document}

\maketitle

\begin{abstract}
This paper is devoted to the construction of weak solutions to the singular constant $Q-$curvature problem. We build on several tools developed in the last years. This is the first construction of singular metrics on closed manifolds of sufficiently large dimension with constant (positive) $Q-$curvature.  
\end{abstract}

\tableofcontents

 \section{Introduction } 
 
 The last years have seen several important works on the $Q-$curvature problem in dimensions bigger or equal to five, since the discovery by Gursky and Malchiodi \cite{GM} of a natural {\sl geometric} maximum principle associated to the Paneitz operator. Building on this work, Hang and Yang \cite{HY1} realized that one could give some conformally covariant conditions under which such a maximum principle holds and provided an Aubin-type result for existence of constant $Q-$curvature metrics on closed manifolds. 
 
 The present paper is devoted to the construction of singular solutions to the constant $Q-$curvature problem for dimensions bigger or equal to five. Before explaining our results, we review the by-now classical setting of the Yamabe problem.  Let $(M,g)$ be a compact $n$-dimensional Riemannian manifold, $n \geq 3$. If $\Sigma \subset M$ is any
closed set, then the standard singular Yamabe problem concerns the existence and geometric properties of
complete metrics of the form $\bar g = u^{\frac{4}{n-2}} g$ with constant scalar curvature. This corresponds to solving
the partial differential equation defined on $M \backslash \Sigma$
\begin{equation}\label{eq:cL}
-\Delta_{g} u + \frac{n-2}{4(n-1)} R^{g} u = \frac{n-2}{4(n-1)}R^{\bar g} \, u^{\frac{n+2}{n-2} }, \qquad u > 0,
\end{equation}
where $R^{\bar g}$ is constant and such that $u(x) \to \infty$ sufficiently quick as $x$ approaches $\Sigma$ so that
$\bar g$ is complete. It is known
that solutions with $R^{\bar  g}< 0$ exist quite generally if $\Sigma$ is large in a capacitary
sense \cite{Lab}, whereas for $R^{\bar g} > 0$ existence is only known when $\Sigma$ is a smooth
submanifold (possibly with boundary) of dimension $k \leq  (n-2)/2$, see \cite{Mazzeo-Pacard96}, \cite{F}.

There are both analytic and geometric motivations for studying this problem. For example, in the positive case ($R > 0$), solutions
to this problem are actually weak solutions across the singular set, so these results fit into the broader investigation of
possibly singular sets of weak solutions of semilinear elliptic equations. On the geometric side is a well-known theorem
by Schoen and Yau \cite{SY} stating that if $(M,g)$ is a compact manifold with a locally conformally flat metric $g$ of
positive scalar curvature, then the developing map $D$ from the universal cover $\widetilde{M}$ to $\mathbb S^n$, which
by definition is conformal, is injective, and moreover, $\Sigma := \mathbb S^n \setminus D(\widetilde{M})$ has Hausdorff
dimension less than or equal to $(n-2)/2$. Regarding the lifted metric $\tilde{g}$ on $\widetilde{M}$, this provides an interesting class of solutions of the singular Yamabe problem which are periodic with respect to a
Kleinian group, and for which the singular set $\Sigma$ is typically nonrectifiable. More generally, that paper also
shows that if $g$ is the standard round metric on the sphere and if $\bar g = u^{\frac{4}{n-2}} g$ is a complete metric
with positive scalar curvature and bounded Ricci curvature on a domain $\mathbb S^n \setminus \Sigma$, then
$\dim_{\mathcal H} \Sigma \leq (n-2)/2$.

In this work, we address the same type of question for the $Q-$curvature equation. The equation involves a fourth order operator, the so-called Paneitz operator, is therefore significantly more challenging to investigate.  However, in the recent years, there has been several new insights on this difficult problem, thanks to the work of Gursky and Malchiodi  \cite{GM}. A major problem for considering higher order equations is the lack of maximum principle. In particular, in general, one cannot ensure that any reasonable approximation yielding to  a weak solution of the equation is actually {\sl positive}. However, the breakthrough of Gursky and Malchiodi ensures that under some {\sl geometric} assumptions on the manifold, one can ensure that one can obtain a positive solution. Their maximum principle led to an existence result similar to the Yamabe problem using a flow approach. A more variational point of view was later implemented by Hang and Yang in \cite{HY1}, weakening also the geometric conditions to have a maximum principle. Keeping in mind the importance of the $Q-$curvature problem both analytically and geometrically, it is then a natural question to ask wether one can construct singular solutions as for the second order case. This is the main result of this paper.

We first describe the setting of our contribution: let $(M,g)$ be a smooth closed $n$-dimensional  Riemannian  manifold with $n\geq 5$. The   $Q$-curvature $Q_g$ is given by \begin{align}  Q_g&=-\frac{1}{2(n-1)}\D R-\frac{2}{(n-2)^2}|Ric|^2+\frac{n^3-4n^2+16n-16}{8(n-1)^2(n-2)^2}R^2\notag\\ &=-\D J-2|A|^2+\frac n2J^2,\end{align} where $R$ is the scalar curvature, $Ric$ is the Ricci curvature tensor, and $$J=\frac{R}{2(n-1)},\quad A=\frac{1}{n-2}(Ric-Jg).$$ The Paneitz operator is given by \begin{align}P\vp&= \D^2\vp+\frac{4}{n-2}div(Ric(\nabla\vp,e_i)e_i)-\frac{n^2-4n+8}{2(n-1)(n-2)}div(R\nabla\vp)+\frac{n-4}{2}Q\vp\notag\\ &=\D^2\vp+div(4A(\nabla\vp,e_i)e_i-(n-2)J\nabla\vp)+\frac{n-4}{2}Q\vp.  \end{align} Here $e_1,\dots,e_n$ is an orthonormal frame with respect to $g$. 
 
 For a given closed sub-manifold $\Sigma$ of $M$, we are interested in finding weak solutions to \begin{align} \label{main-1}Pu=u^\frac{n+4}{n-4}\quad\text{in }M\setminus \Sigma,\end{align} such that $u$ goes to infinity as one approaches $\Sigma$, that is, for every $p\in\Sigma$ and $x_k\to p$ with $x_k\in M\setminus\Sigma $, $u(x_k)\to\infty$. 
 
 Our main result is the following.

\begin{thm}\label{main-theorem}
Let $\Sigma$ be a connected smooth closed  ( i.e. compact without boundary)  submanifold of $M$. Assume that $Q\geq 0$, $Q\not\equiv0$ and $R\geq0$.  If    $0<dim(\Sigma)=:k<\frac{n-4}{2}$ then there exists an infinite dimensional family of complete metrics on  $M \backslash \Sigma$ with constant $Q-$curvature.   \end{thm}

Actually, Theorem \ref{main-theorem} is a corollary of the following result.

\begin{thm}\label{main2-theorem}
Let  $\Omega \subset \R^n$ be a smooth open set and $\Sigma= \cup_{i=1}^K \Sigma_i$ a disjoint union of smooth, closed submanifolds of dimensions $k_i$ in $\Omega$. Assume that $n$ and $k_i$ for $i=1,..., K$ satisfy
$$
\frac{n-k_i}{n-k_i-4}<p <  \frac{n-k_i+4}{n-k_i-4}
$$
or equivalently 
$$
n-\frac{4p+4}{p-1}<k_i<n-\frac{4p}{p-1}. 
$$
Then there exists a positive weak solution to 
\begin{align} \label{eq-domain} \left\{\begin{array}{ll} \D^2u=u^p &\quad\text{in }\Omega \backslash \Sigma \\  u=\D u=0&\quad\text{on }\partial\Omega\\u>0&\quad\text{in }\Omega,\end{array}  \right.\end{align}
that blows up exactly at $\Sigma$. Furthermore if at least one of the $k_i >0$, then the solution space  for \eqref{eq-domain} is  infinite dimensional.
\end{thm}

\begin{rem}
Notice that in the previous theorem if $k_i=0$ for all $i=1,...,K $, then the exponent $p$ has to be subcritical with respect to the Sobolev exponent and supercritical with respect to the Serrin's exponent provided $n \geq 5$. On the other hand, if one of the $k_i$ is positive, then the critical exponent $\frac{n+4}{n-4}$ is allowed for $p$, but the dimension $n$ has to be large enough. 

\end{rem}

The proof of Theorem \ref{main-theorem} uses several tools ranging from geometric theory of edge operators (as in \cite{Mazzeo-Pacard96}) to a more general view point on this type of problem provided in \cite{Ao1}. Since we are dealing with a fourth order equation, even the ODE analysis, which is instrumental in \cite{Mazzeo-Pacard96}, is rather involved. On the other hand, the authors in \cite{Ao1} had to develop ODE-free method to deal with their quite general operators. Using the model  $\mathbb R^n \backslash \mathbb R^k$ which is conformally equivalent to the product $\mathbb S^{n-k-1} \times \mathbb H^{k+1}$ with the canonical metric, a straightforward computation of the $Q-$curvature on this model provides the condition $0 < k <\frac{n-4}{2}$ for positive $Q-$curvature metrics (see e.g. \cite{BPS}). This model plays a crucial role in our theory since it allows to by-pass some tricky ODE arguments by having an``explicit" form of the solution using the Fourier-Helgason transform on hyperbolic space. See \cite{CHY} for much deeper results related to the dimension restriction and \cite{BPS} for multiplicity results on the Q-curvature problem. Finally, as in the second order case (Yamabe problem), we use Delaunay-type solutions as building blocks of the approximate solution (see \cite{GWZ} and the Appendix for some existence results on these solutions). Note also that \cite{AB} provides also singular solutions using the trivial profile  $|x|^{-\frac{4}{p-1}}$ but those allow to build only local solutions (see e.g. \cite{Mazzeo-Smale} for the Yamabe case).

\section{Preliminaries}

\subsection{Function spaces}

Let $\Sigma$ be a smooth $k$ dimensional submanifold of $\Omega \subset \R^n$ (or a union of submanifolds with different dimensions). For $\sigma>0$ small we let $N_\sigma$ to be the geodesic tubular neighborhood of radius $\sigma$ around $\Sigma$.  For $\alpha\in(0,1)$, $s\in (0,\sigma)$, $k\in \N\cup\{0\}$ and $\nu\in\R$ we define the seminorms \begin{align} \notag|w|_{k,\alpha,s}:=\sum_{j=0}^k s^j\sup_{N_s\setminus N_\frac s2}|\nabla ^jw| +s^{k+\alpha}\sup_{x,x'\in N_s\setminus N_\frac s2} \frac{|\nabla^kw(x)-\nabla^kw(x')|}{|x-x'|^\alpha}, \end{align} and  the weighted H\"older norm\begin{align} \notag \|w\|_{C^{k,\alpha}_\nu}  :=|w|_{C^{k,\alpha}(\bar\Omega\setminus N_\frac\sigma2)} +\sup_{0<s<\sigma}s^{-\nu}|w|_{k,\alpha,s}.\end{align} The  weighted H\"older  space $C^{k,\alpha}_\nu(\Omega\setminus\Sigma)$ is defined by \begin{align} \notag C^{k,\alpha}_\nu(\Omega\setminus\Sigma):=\left\{w\in C^{k,\alpha}_{loc}(\bar \Omega\setminus\Sigma): \|w\|_{C^{k,\alpha}_\nu} <\infty \right\}. \end{align}  The subspace of $C^{k,\alpha}_\nu(\Omega\setminus\Sigma)$ with Navier boundary conditions will be denoted by $$C^{k,\alpha}_{\nu,\mathscr{N}}(\Omega\setminus\Sigma):=\{w\in C^{k,\alpha}_\nu(\Omega\setminus\Sigma):w=\D w=0\text{ on }\partial\Omega\}.$$
The space $C^{k,\alpha}_{\nu,\nu'}(\R^N\setminus \{0\})$ is defined by  \begin{align}\notag \|w\|_{C^{k,\alpha}_{\nu,\nu'}(\R^N\setminus \{0\})}:=\|w\|_{C^{k,\alpha}_{\nu}(B_2\setminus \{0\})} +\sup _{r\geq 1}(r^{-\nu'} \|w(r\cdot)\|_{C^{k,\alpha}(\bar B_2\setminus B_1)}).\end{align} We also set  \begin{align}\notag  \|w\|_{C^{k,\alpha}_{\nu,\nu'}(\R^n\setminus \R^m)}:=\|w\|_{C^{k,\alpha}_{\nu}(\mathcal{B}_2\setminus \R^m)} +\sup _{r\geq 1}(r^{-\nu'} \|w(r\cdot)\|_{C^{k,\alpha}(\bar{ \mathcal{B}}_2\setminus \bar{\mathcal{B}}_1) },\end{align}   where $\mathcal{B}_r$ denotes the tubular neighborhood of radius $r$ of $\R^m$ in $\R^n$. \\

We now list some useful properties of the space $C^{k,\alpha}_{\nu}(\Omega\setminus\Sigma)$, see e.g. \cite{Mazzeo-Pacard96} and the book \cite{Pacard-Riviere}.

\begin{lem} \label{lem-properties}The following properties hold: \begin{itemize} \item[i)] If $w\in C^{k+1,\alpha}_{\gamma}(\Omega\setminus\Sigma)$ then   $\nabla w\in C^{k,\alpha}_{\gamma-1}(\Omega\setminus\Sigma)$. 
\item[ii)] If  $w\in C^{k+1,0}_{\gamma}(\Omega\setminus\Sigma)$ then   $ w\in C^{k,\alpha}_{\gamma}(\Omega\setminus\Sigma)$ for every $\alpha\in[0,1)$. 
\item[iii)]  For every $w_i\in C^{k,\alpha}_{\gamma_i}(\Omega\setminus\Sigma)$,\, i=1,2, we have 
 $$\|w_1 w_2\|_{k, \gamma_1+\gamma_2,\alpha}\leq C \|w_1 \|_{k, \gamma_1,\alpha} \|w_2\|_{k, \gamma_2,\alpha},$$ for some $C>0$ independent of $w_1,w_2$.
\item[iv)] There exists $C>0$ such that for every   $w\in C^{k,\alpha}_{\gamma}(\Omega\setminus\Sigma)$   with $w>0$ in $\bar \Omega\setminus \Sigma$ we have  $$\|w^p\|_{k, \gamma,\alpha}\leq C \|w \|^p_{k, \gamma,\alpha} .$$

 \end{itemize}  \end{lem}
 
 \subsection{Fermi coordinates}

We now compute the Fermi coordinates for our problem. For $\sigma>0$ small we can choose fermi coordinates in $N_\sigma$ as follows:  First we fix any local coordinate system $y=(y_1,\dots,y_k)$ on $\Sigma$ ($k$ is the dimension of $\Sigma$). For every $y_0\in\Sigma$ there exists an orthonormal frame field $E_1,\dots,E_{n-k}$, basis of the normal bundle of $\Sigma$. Then we consider the coordinate system $$\Sigma\times\R^{n-k}\ni (y,z)\to y+\sum z_iE_i(y).$$ For $|z|<\sigma$ with $\sigma$ small, these are well-defined coordinate system in a neighborhood of $y_0$.  
 
 In this coordinate system the Euclidean metric has the following expansion $$g_{\R^n}=g_{\R^{n-k}}+g_\Sigma +O(|z|)dzdy+O(|z|)dy^2.$$
 Therefore, \begin{align} \notag \D_{\R^n} =\D_{\R^{n-k}}+\D_\Sigma  +\mathsf{e_1}\nabla+\mathsf{e_2} \nabla^2, \end{align} where $\mathsf{e_i}$, $i=1,2$ satisfy $$\|\mathsf{e_1}\|_{C^{0,\alpha}_0}+\|\mathsf{e_2}\|_{C^{0,\alpha}_1}\leq c. $$   Using this we also have \begin{align}\label{Delta}\D^2_{\R^n} =\D^2_{\R^{n-k}}+\D^2_\Sigma  +2\D_{\R^{n-k}}\D_\Sigma  +\sum_{i=1}^4\mathsf{e_i}\nabla^i ,\end{align} where \begin{align} \label{est-13}\|\mathsf{e_1}\|_{C^{0,\alpha}_{-2}}+\|\mathsf{e_2}\|_{C^{0,\alpha}_{-1}}+\|\mathsf{e_3}\|_{C^{0,\alpha}_0}+\|\mathsf{e_4}\|_{C^{0,\alpha}_1}\leq c.\end{align}
 
 \subsection{The singular solution}
 
 The building block for our theory is the existence of a singular solution with different behaviour at the origin and at infinity. The following theorem provides such a solution. We refer the reader to the appendix for a proof of this result. 
 
  \begin{thm}\label{exists1}
  Let $N \geq 5$. % and $\Omega \subset \R^N$ a smooth open set. 
  Suppose that $\frac{N}{N-4}<p<\frac{N+4}{N-4}$. Then for every $\beta>0$ there exists a unique radial solution $u$ to \begin{align} \label{singu11} \left\{\begin{array}{ll} \D^2u=u^p\quad\text{in }\R^N\setminus\{0\}\\ u>0 \quad\text{in }\R^N\setminus\{0\}\\
  \lim_{|x|\to0} u(x)=\infty,  \end{array}  \right.\end{align}  such that $$\lim_{r\to\infty}r^{N-4}u(r)=\beta,\quad \lim_{r\to 0^+}r^\frac{4}{p-1}u(r)=c_p:=[k(p,N)]^\frac{1}{p-1},$$ where \begin{align*}k(p,N)&=\frac{8(p+1)}{(p-1)^4}\left[ N^2(p-1)^2+8p(p+1)+N(2+4p-6p^2) \right]. \end{align*} \end{thm}

\subsection{Approximate solutions} 

Let $u$ be a singular radial solution to \eqref{singu11}. Then $u_\ve(x):=\ve^{-\frac{4}{p-1}}u(\frac x\ve)$ is also a solution to \eqref{singu11}. Note that $$u_\ve(x)\leq C(\delta,u) \ve^{N-4-\frac{4}{p-1}}\quad\text{for }|x|\geq\delta, $$ which shows that $u_\ve\to0$ locally uniformly in $\R^N\setminus\{0\}$. Due to this scaling and the asymptotic behavior of $u$ at infinity, for a given $\alpha>0$,  we can find a solution $u_1$ such that $r^4u_1^{p-1}(r)\leq \alpha$ on  $(1,\infty)$.

\subsubsection{Isolated singularities}  

Let $\Sigma=\{x_1,x_2,\dots,x_K\}$ be a set of finite points in $\Omega$. To construct a solution to \eqref{eq-domain} which is singular precisely at the points of $\Sigma$, we start by constructing  an approximate solution to \eqref{eq-domain} which is singular exactly on $\Sigma$. Let us first fix a smooth cut-off function $\chi$ such that $\chi=1$ on $B_1$ and $\chi=0$ on $B_2^c$. Also fix $R>0$  such that $B_{2R}(x_i)\subset\Omega$ and $B_{2R}(x_i)\cap B_{2R}(x_j)=\emptyset$ for every $i\neq j$, $i,j\in\{1,2,\dots,K\}$. Let $\bar\ve=(\ve_1,\ve_2,\dots,\ve_K)$ be a $K$-tuple of dilation parameter.  An approximate solution $\bar u_{\bar\ve}$ is defined by $$\bar u_{\bar \ve}(x)=\sum_{i=1}^K\chi_R(x-x_i)u_{\ve_i}(x-x_i)=\sum_{i=1}^K \ve_i^{-\frac{4}{p-1}}\chi(\frac{x-x_i}{R})u_1(\frac{x-x_i}{\ve_i}) .$$ The asymptotic behavior of $u_1$ at infinity  leads to the following error of approximation:
\begin{lem}  The error  $f_\bve:=\D^2 \bar u_\bve- \bar u^p_\bve$ satisfies $$\|f_\bve\|_{C^{0,\alpha}_{\gamma-4}}\leq C_\gamma \ve_0^{N-\frac{4p}{p-1}}\quad\text{for } 0<\ve_i\leq\ve_0\leq1,$$ for every $\gamma\in\R$. \end{lem}

\subsubsection{Higher dimensional singularities} \label{section-higher}

Let $\Sigma_i\subset \Omega$ be a $k_i$-dimensional submanifold in $\Omega$ for $i=1,2,\dots,K$. We fix $\sigma>0$ small such that  Fermi coordinates are well-defined  on  the Tubular  neighborhood $N_{i,\sigma}$  of $\Sigma_i$ for every $i=1,\dots,K$, and $N_{i,2\sigma}\cap N_{j,2\sigma}=\emptyset$ for $i\neq j$. Fix a smooth radially symmetric  cut-off function $\chi$ such that $\chi=1$ on $B_1$ and $\chi=0$ on $B_2^c$. Then for $0<\ve_i<1$ and $0<R<\frac\sigma2$ we set $$\bar u_{\ve_i}(x,y)=\ve_i^{-\frac{4}{p-1}}u_1(\frac {x}{\ve_i})\chi(\frac yR)=:u_{\ve_i}(x)\chi_R(y).$$ An approximate solution which is singular only on $\cup_{i=1}^K\Sigma_i$ is defined by $$\bar u_{\bar \ve}=\sum_{i=1}^K\bar u_{\ve_i}.$$ Using the expansion \eqref{Delta} and the estimate \eqref{est-13} we see that the error $f_\bve:=\D^2\bar u_\bve-\bar u_\bve^p$ satisfies $$\|f_\bve\|_{C^{0,\alpha}_{\gamma-4}}\leq C\ve_0^q,\quad 0<\ve_i\leq \ve_0\leq1,\, q:=\frac{p-5}{p-1}-\gamma>0,$$ for $\gamma< \frac{p-5}{p-1}$. In our applications, $\gamma$ will be bigger than  $ -\frac{4}{p-1}$.

\begin{rem}
To prove existence of solution to \eqref{eq-domain} with singular set $\Sigma$, we shall look for solutions of the form $u=\bar u_\bve+v$ (in both cases, that is, $\Sigma$ is finite and higher dimensional). Then, $v$ has to satisfy \begin{align}\label{eq-v} L_\bve v+f_\bve+Q(v)=0,\quad L_\bve:=\D^2-p\bar u_\bve^{p-1},  \end{align}  where \begin{align}\label{Q}Q(v)=-(\bar u_\bve+v)^p+\bar u_\bve^p+p\bar u_\bve^{p-1} v . \end{align}
\end{rem}

 \medskip

 \section{Linearized operator in $\R^N \backslash \left \{ 0 \right \}$}
 
 Since our purpose is to use an implicit function theorem, it is crucial to understand the linearized problem. For this, we invoke the analytic theory of edge operators as in \cite{mazzeoEdge,mazzeoVertman} but also some more general arguments in \cite{Ao1} as we mentioned in the introduction.

We consider the linearized operator $$L_1=\D^2-pu_1^{p-1}$$ where in polar coordinates we denote $$\D=\frac{\partial^2}{\p r^2}+\frac{N-1}{r}\frac{\p}{\p r}+\frac{1}{r^2}\D_\theta.$$  %$$\D^2=$$ 

\subsection{Indicial roots}
 
Next we compute indicial roots of the linearized operator $L_1$. We recall that   $\gamma_j$ is an indicial root of $L_1$ at $0$ if $L_1(|x|^{\gamma_j}\vp_j)=o(|x|^{\gamma_j-4})$, where $\vp_j$ is the $j$-th eigenfunction of $-\D_\theta$, that is $-\D_\theta\vp_j=\lambda_j\vp_j$, $$\lambda_0=0,\quad \lambda_j=N-1,\quad \text{for }j=1,\dots,N,$$ and so on.  One shows  that $\gamma_j$ is a solution to $$[\gamma(\gamma-1)+(N-1)\gamma-\lambda_j][(\gamma-2)(\gamma-3)+(N-1)(\gamma-2)-\lambda_j]-A_p=0,$$ where \begin{align}\label{Ap}A_p:=p\lim_{r\to0}r^4u_1^{p-1}(r)=pk(p,N).\end{align} The solutions are given by  \begin{align*}\gamma_j^{\pm\pm} =\frac12\left[4-N\pm\sqrt{(N-2)^2+4+4\lambda_j\pm 4\sqrt{(N-2)^2+4\lambda_j+A_p}}\right] .\end{align*}
 We have that  ($\Re$ denotes the real part) \begin{align}\label{indicial18}\gamma_0^{-+}<4-N<-\frac{4}{p-1}<\Re(\gamma_0^{--})\leq \frac{4-N}{2}\leq \Re(\gamma_0^{+-})<0<2<\gamma_0^{++},  \end{align}  and \begin{align} \gamma_j^{-\pm}<-\frac{4}{p-1},\quad \Re(\gamma_0^{+-})<\gamma_j^{+\pm}\quad\text{for }j\geq1. \end{align}
To prove the  above relations one uses that $A_p$ is monotone. Indeed,   for any fixed $N\geq 5$, $\frac{\p}{\p p}A_p$  vanishes at the following points  $$p_0:=\frac{N+2}{N-6},\quad p_1^\pm= \frac{N+4\pm 2\sqrt{N^2+4}}{3N-8}.$$   Using this one would get that $A_p$ is monotone increasing on $(\frac{N}{N-4},\frac{N+4}{N-4})$. 
 
 Since $\lim_{r\to\infty}r^4u_1^{p-1}(r)=0$, the indicial roots of $L_1$ at infinity are the same as for the $\D^2$ itself. These values are given by  \begin{align*}\tilde \gamma_j^{\pm\pm} =\frac12\left[4-N\pm\sqrt{(N-2)^2+4+4\lambda_j\pm 4\sqrt{(N-2)^2+4\lambda_j}}\right] .\end{align*} In particular, 
 $$\tilde\gamma_0^{\pm\pm}\in\{2-N,4-N,0,2 \},\quad \tilde\gamma_1^{+\pm}\in\{1,3\},\quad  \tilde \gamma_j^{+\pm}\geq 1\quad\text{and }\tilde \gamma_j^{-\pm}<4-N\quad\text{for }j\geq 1.$$ %$\tilde\gamma_1^{+\pm}\in\{1,3\}$.

 We shall choose     $\mu,\nu$ in the region   \begin{align}\label{mu-nu} \frac{-4}{p-1}<\nu<\Re(\gamma_0^{--})\leq\frac{4-N}{2}\leq\Re(\gamma_0^{+-})<\mu 
 ,  \end{align} so that $\mu+\nu=4-N$.

\subsection{Study of the linearized operator}
 
 For a function  $w=w(r,\theta)$ we decompose it as \begin{align}\label{decompose}w(r,\theta)=\sum_{j=0}^\infty w_j(r)\vp_j(\theta). \end{align} Then \begin{align*}\D^2 (w_j\vp_j)&=\left(w_j^{iv}+\frac{2(N-1)}{r}w_j'''+\frac{N^2-4N+3-2\lambda_j}{r^2}w_j''-\frac{(N-3)(N-1+2\lambda_j)}{r^3}w_j'\right. \\ &\quad  \left.+\frac{2(N-4)\lambda_j+\lambda_j^2}{r^4}w_j\right)\vp_j\\ &= : (w_j^{iv}+\frac{a_{1,j}}{r}w_j'''+\frac{a_{2,j}}{r^2}w_j''-\frac{a_{3,j}}{r^3}w_j'+\frac{a_{4,j}}{r^4}w_j)\vp_j.
 %(\D_r-\frac{\lambda_j}{r^2})^2w_j\vp_j 
  \end{align*}
Thus, $L_1 w=0$ if and only if,  for every $j=0,1,2,\dots$ \begin{align}\label{eq-wj-22}w_j^{iv}+\frac{a_{1,j}}{r}w_j'''+\frac{a_{2,j}}{r^2}w_j''-\frac{a_{3,j}}{r^3}w_j'+\frac{a_{4,j}-V_p(r)}{r^4}w_j=0,\quad V_p(r):=pr^4u_1^{p-1}(r). \end{align} 
%Using the relations  \begin{align*} r^{N-1}ww^{iv}= (r^{N-1}ww''')'-(N-1)r^{N-2}ww'''-r^{N-1}w'w'''  \\
 %r^{N-2}ww'''= (r^{N-2}ww'')'-(N-2)r^{N-3}ww''-r^{N-2}w'w''  \\
% r^{N-3}ww''= (r^{N-3}ww')'-(N-3)r^{N-4}ww'-r^{N-3}(w')^2  \\
% r^{N-1}w'w'''= (r^{N-1}w'w'')'-(N-1)r^{N-2}w'w''-r^{N-1}(w'')^2  \\
% r^{N-2}w'w''= (r^{N-3}ww')'-(N-3)r^{N-4}ww'-r^{N-3}(w')^2  \\
% \end{align*}
One obtains 
 \begin{align} &r^{N-1}w_j(w_j^{iv}+\frac{a_{1,j}}{r}w_j'''+\frac{a_{2,j}}{r^2}w_j''-\frac{a_{3,j}}{r^3}w_j' +\frac{a_{4,j}}{r^4}w_j)\notag\\ 
 &=(r^{N-1}w_jw'''_j)'   -(r^{N-1}w_j'w_j'')'+(N-1)(r^{N-2}w_jw_j'')' - (N-1+2\lambda_j)(r^{N-3}w_jw_j')' \notag\\ &\quad  +(N-1+2\lambda_j)r^{N-3}(w')^2+  r^{N-1}(w'')^2+a_{4,j}r^{N-5}w_j^2.  \label{byparts25}
  \end{align}

  \begin{prop} \label{inj-prop-1}  Let $w\in C^{4,\alpha}_{\mu,0}(\R^N\setminus\{0\})$  be a solution to $L_1w=0$. Then $w\equiv0$. \end{prop}
 \begin{proof} From the definition of the space $C^{4,\alpha}_{\mu,0}(\R^N\setminus\{0\})$ we have that   
 \begin{align}\label{asymp-w-24} \left\{  \begin{array}{ll}|w(x)|\leq C\log(2+|x|),\quad |x|^k |\nabla ^kw(x)|\leq C &\quad\text{for }|x|\geq 1,\quad k=1,\dots,4\\  \rule{0cm}{.5cm}|x|^{-\mu+k}|\nabla^k w(x)|\leq C &\quad\text{for }0<|x|\leq1, \quad k=0,\dots,4.\end{array}\right. \end{align}  Now we decompose $w$ as in \eqref{decompose}.
 
 First we show that $w_0=0$. From the choice of $\mu$ we see that if $w_0\not\equiv 0$ then $w_0$ should behave like $r^{\gamma_0^{++}}$ around the origin.  Therefore, without loss of any generality,  we can assume that $w_0\geq 0$ in a small neighborhood of the origin. Using the crucial fact $q:=\gamma_0^{++}>2$, thanks to  \eqref{indicial18},   we shall show that $w$ is actually $C^2$ and $\D w_0(0)=0$. Indeed, as $w_0$ satisfies $\D^2w_0=pu_1^{p-1}w_0=:f$ in $\R^N\setminus\{0\}$,  for $0<\ve<r$ we have $$(\D w_0)'(r)r^{N-1}=(\D w_0)'(\ve)\ve^{N-1}+\omega_N\int_{\ve<|x|<r}fdx,  \quad\omega_N:=|S^{N-1}|^{-1}.$$ As $\mu>(4-N)/2$ we see that $$|(\D w_0)'(\ve)|\ve^{N-1}\leq C\ve^{\mu-3}\ve^{N-1}\to0\quad\text{as }\ve\to0.$$ Hence, for $0<r_1<r_2$ we get  $$\D w_0(r_2)=\D w_0(r_1)+\omega_N\int_{r_1}^{r_2}\frac{1}{t^{N-1}}\int_{|x|<t}f(x)dxdt.$$ As $(\D w_0)'>0$ on $(0,\ve_0)$ for some $\ve_0>0$, that is $\D w_0$ is monotone increasing, we see from the above relation that $\lim_{r_1\to0^+}\Delta w_0(r_1)$ exists and finite. Here we used that $f(x)\approx |x|^{q-4}$, $q>2$. Thus, $w_0$ is $C^2$, and again as $q>2$, we must have $\D w_0(0)=0$.  In conclusion, $\D w_0(r)\geq \D w_0(1)>0$ for $r>1$, which leads to $w_0(r)\gtrsim r^2 $, a contradiction to \eqref{asymp-w-24}.

 For $j\geq 1$,  $w_j$ behaves like $r^{4-N-\tilde\gamma_j^{\pm\pm}}$ as $r\to\infty$. Since $\tilde \gamma_j^{-\pm}<4-N$ and $w_j=O(\log r) $ at infinity,  we must have $w_j\approx  r^{4-N-\tilde\gamma_j^{+\pm}}$. % {\color{blue}  need to show that derivatives have the right decay.}
 
 Now multiplying the equation \eqref{eq-wj-22}    by $r^{N-1}w_j$, and using  \eqref{byparts25} we get (this is justified thanks to \eqref{asymp-w-24}, and the asymptotic behavior of $w_j$ at infinity) \begin{align*} (N-1+2\lambda_j) &\int_0^\infty r^{N-3}w'^2dr+\int_0^\infty r^{N-1}w''^2dr=\int_0^\infty [V_p(r)-2(N-4)\lambda_j-\lambda_j^2]r^{N-5}w_j^2dr \\ &\leq  \left[p\frac{p+1}{2} k(p,N)-2(N-4)\lambda_j-\lambda_j^2\right]\int_0^\infty r^{N-5}w_j^2 dr \\ &\leq   \left[ \frac{1}{16}N^3 (N+4)-2(N-4)\lambda_j-\lambda_j^2\right]\int_0^\infty r^{N-5}w_j^2dr\\ &=:C(N,j)\int_0^\infty r^{N-5}w_j^2dr. \end{align*}
 An integration by parts gives  $$\int_0^\infty r^{N-5}w_j^2dr\leq\frac{4}{(N-4)^2}\int_0^\infty r^{N-3}w_j'^2dr$$ 
 $$\int_0^\infty r^{N-3}w_j'^2dr\leq\frac{4}{(N-2)^2}\int_0^\infty r^{N-1}w_j''^2dr.$$ This leads to \begin{align*} \int_0^\infty r^{N-1}w_j''^2dr &\leq  \left[\frac{4   C(N,j)}{(N-4)^2} - (N-1+2\lambda_j) \right] \int_0^\infty r^{N-3}w_j'^2dr \\  & \leq  \left[\frac{4   C(N,j)}{(N-4)^2} - (N-1+2\lambda_j) \right]  \frac{4}{(N-2)^2}\int_0^\infty r^{N-1}w_j''^2dr  \\ &=:\bar C(N,j)\int_0^\infty r^{N-1}w_j''^2dr.\end{align*} One can show that $\bar C(N,j)<1$ for $j\geq N+1$, and hence $w_j\equiv0$ for $j\geq N+1$. 
 
 Finally,  we consider the case  $j=1,\dots N$.  For $\lambda_1=N-1$ we have that $\tilde\gamma_1^{+\pm}\in \{1,3\}$, and $\tilde\gamma_1^{-\pm} <4-N$. Therefore, $w_1$ should behave like $r^{1-N}$ or $r^{3-N}$ at infinity.  
 
 Let us first show that if $w_1=O(r^{1-N})$ at infinity and $w_1\vp_1\in C^{4,\alpha}_{\mu,0}$  then $w_1\equiv 0$.  
 
 We set $$W(r)=-\int_{r}^\infty w_1(t)dt, \quad r>0,$$ so that $W'=w_1$. %We also set  (recall that $\vp_1=\frac{x_1}{|x|}$) $$w=w_1\vp_1=\frac{\partial }{\partial x_1}W(|x|), \quad v_\rho(x)=\rho u_1'(|x|)\frac{x_1}{|x|}= \rho \frac{\partial }{\partial x_1}u_1(|x|).$$ 
 For $\ve>0$ small let $\Omega_\ve$ be the domain $$\Omega_\ve:=\left\{x\in\R^N:x_1>0,\, |x|>\ve  \right\}.$$  
 
 We know that $u'_1<0$, $\D u_1<0$, $(\D u_1)'>0 $ on $(0,\infty)$, % (estimates on the derivatives of $u$ are not given in \cite{GWZ}) 
 $$u_1(r)\approx r^{4-N},\quad u'_1 (r)\approx -r^{3-N},\quad  \D u_1(r)\approx -r^{2-N},  \quad (\D u_1)'(r)\approx r^{1-N} \quad\text{as }r\to\infty ,$$ and $$u_1(r)\approx r^{-\frac{4}{p-1}},\quad u_1' (r)\approx -r^{-\frac{4}{p-1}-1},\quad  \D u_1(r)\approx -r^{-\frac{4}{p-1}-2},  \quad (\D u_1)'(r)\approx r^{-\frac{4}{p-1}-3} \quad\text{as }r\to 0.$$  Therefore, $$W(r)=o(u_1(r)),\quad W'(r)=o(u_1'(r)), \quad (\D W)'(r)=o((\D u_1)'(r))\quad\text{as }r\to0\text{ or }\infty.$$ Setting  $$ \tilde w_\rho(x)=(W'(|x|)-\rho u_1'(|x|))\frac{x_1}{|x|}=\frac{\partial }{\partial x_1}(W-\rho u_1)(|x|),$$ we see that for $\rho>>1$ we have \begin{align*} \tilde   w_\rho\geq 0\quad \text{and }\quad \D \tilde  w_\rho=(\D W-\rho\D u_1)'(x)\frac{x_1}{|x|}\leq 0\quad \text{in }\Omega_\ve,\end{align*} equivalently  \begin{align}\label{w-rho} W'-\rho u_1'\geq 0\quad \text{and }\quad (\D W)'-\rho(\D u_1)'\leq 0\quad \text{in }\Omega_\ve.\end{align} 
  Now we set $$\rho_\ve:=\inf\{\rho>0:\eqref{w-rho}\text{ holds} \}.$$ We claim that $\rho_\ve\to0$ as $\ve\to0$. Indeed, as $\tilde w_{\rho_\ve}$ satisfies (recall that $\vp_1=\frac{x_1}{|x|}$) $$\D^2\tilde  w_{\rho_\ve}=pu^{p-1}\tilde w_{\rho_\ve}\geq0,\quad \D \tilde w_{\rho_\ve}\leq 0 \text{ in }\Omega_\ve,$$ by maximum principle \begin{align}
  \label{maximum}\tilde w_{\rho_\ve}>0\quad\text{and }\D\tilde  w_{\rho_\ve}<0\quad\text{in }\Omega_\ve.\end{align} On the other hand, if $\rho_\ve>0$ then  there exists $x_\ve\in\bar\Omega_\ve$ such that $$W'(x_\ve)-\rho_\ve u'_1(x_\ve) = 0\quad \text{ or  }\quad (\D W)'(x_\ve)-\rho_\ve(\D u_1)'(x_\ve)=0,$$ thanks to the definition of $\rho_\ve$ and the asymptotic behavior of $W',(\D W)', u_1', (\D u_1)'$. Since $W$ and $u_1$ are radially symmetric,  \eqref{maximum} implies that $|x_\ve|=\ve$. Hence, from the behavior of $W',(\D W)', u_1', (\D u_1)'$ around the origin, we conclude that $\rho_\ve\to0$.

 Thus we have shown that $W'\geq 0$ on $(0,\infty)$. In a similar way,  taking $$\tilde  w_\rho(x)=(-W'(|x|)-\rho u_1'(|x|))\frac{x_1}{|x|}$$ we would get that $W'\leq 0$ on $(0,\infty)$. This completes the roof. 
 
The same proof shows that there is no solution $w_1$ such that  $w_1(r)=u_1'(r)(1+o(1))$ around the origin, and $w_1(r)=o(u_1'(r))$ at infinity. 
 
 Now we show that  if $w_1=O(r^{3-N})$ at infinity and $w_1\vp_1\in C^{4,\alpha}_{\mu,0}$  then $w_1\equiv 0$.  Indeed, if  $w_1\not\equiv 0$, then $\tilde w_1:=u_1'+a w_1$ would satisfy  $\tilde w_1=u'_1(1+o(1))$ around the origin and $\tilde w_1=o(u_1')$ at infinity   for some non-zero constant $a$, which is a contradiction.  
 
 This finishes the lemma. 
  \end{proof} 
  
\begin{prop}\label{inj-prop-2}Let $w\in C^{4,\alpha}_{\mu,\mu}(\R^N\setminus\{0\})$ be a solution to $$\D^2 w-\frac{A_p}{r^4}w=0\quad\text{in }\R^N\setminus\{0\},$$ where $A_p$ is given by \eqref{Ap}. Then $w\equiv 0$.  \end{prop}
  \begin{proof} Note that in this case also we have same indicial roots as before. Therefore, $w$ is given by  $$w=\sum_{j\geq 0}(c_{1,j}r^{\gamma_j^{++}}+c_{2,j}r^-{\gamma_j^{+-}}+c_{3,j}r^{\gamma_j^{-+}}+c_{4,j}r^{\gamma_j^{--}})\vp_j,$$  for some  $c_{i,j}\in\R$. Since,   $r^{\gamma_j^{\pm\pm}}$ is not bounded by $r^{\mu}$ simultaneously  at the origin and at infinity,  we have that $c_{i,j}=0$ for every $(i,j)$. This finishes the lemma.   \end{proof}
  
  \medskip 
  
  \section{Linearized operator  on $\R^n\setminus\R^k$}

  The main goal of this section is to prove the following injectivity of the linearized operator on $\R^n\setminus\R^k$. We closely follow the approach in \cite{Ao1}. 
  
  \begin{prop}\label{inj-prop-3} The linearized operator   $$\mathbb{L}_1= \D^2_{x,y}-pu_1^{p-1}$$ is injective in $C^{4,\alpha}_{\mu,0}(\R^n\setminus\R^k)$. \end{prop}

  In order to prove the above proposition we shall use our previous injectivity results on $\R^N\setminus\{0\}$. The idea is to show that both operators have the same symbol. To be more precise, we write the Euclidean metric in $\R^N$   as $$|dx|^2=dr^2+r^2 g_{\mathbb  S^{N-1}}.$$ We consider the conformal change $$g_0:=\frac{1}{r^2}|dx|^2=dt^2+g_{\mathbb S^{N-1}},\quad r=e^{-t},$$ which is a complete metric on the cylinder $\R\times \mathbb S^{N-1}$. Then the conformal Laplacian $P^{g_0}_{\gamma}$ of order $2\gamma$ with $0<\gamma<\frac N2$ is given by $$P^{g_0}_\gamma w=r^\frac{N+2\gamma}{2}(-\D)^\gamma u,\quad u:=r^{-\frac{N-2\gamma}{2}}w. $$  In the following we shall use the following normalization  on the definition of Fourier transformation on $\R$: $$\hat w(\xi):=\frac{1}{\sqrt{2\pi}}\int_{\R}e^{-i\xi t}w(t)dt.$$

The following lemma can be found in \cite{mar}. 
  
  \begin{lem}Let $P^{j}_\gamma$ be the projection of the operator $P^{g_0}_\gamma$ on the eigenspace $\langle\vp_j\rangle$. Then,  writing $w(t,\theta)=\sum_{j=0}^\infty w_j(t)\vp_j(\theta)$ we have  $$\widehat{P^j_\gamma w_j}=\Theta^j_\gamma(\xi)\hat w_j,$$ where the Fourier symbol is given by \begin{align}\label{symbol-1}\Theta^j_\gamma(\xi)=2^{2\gamma}\frac{|\Gamma(\frac12+\frac\gamma2+\frac12\sqrt{(\frac N2-1)^2+\lambda_j}+\frac\xi2 i)|^2}{|\Gamma(\frac12-\frac\gamma2+\frac12\sqrt{(\frac N2-1)^2+\lambda_j}+\frac\xi2 i)|^2}.\end{align}  \end{lem}
  
  Now we move on to the case when the singularity is along $\R^k$.   For a point $z=(x,y)\in\R^n\setminus\R^k$ we shall use the following notations: $x\in \R^N$, $y\in \R^k$ where $\R^n=\R^N\times\R^k$. We shall  also write $z=(x,y)=(r,\theta,y)$ where $r=|x|$ and $\theta\in \mathbb S^{N-1}$.   Then the Euclidean metric on $\R^n$ can be written as $$|dz|^2=|dx|^2+|dy|^2=dr^2+r^2g_{\mathbb S^{N-1}}+dy^2.$$ Now we consider the conformal metric $$g_k:=\frac{1}{r^2}|dz|^2=g_{\mathbb S^{N-1}}+\frac{dr^2+dy^2}{r^2}=g_{\mathbb S^{N-1}}+g_{\HH^{k+1}},$$ where $\HH^{k+1}$ is the Hyperbolic space.  The conformal Laplacian is given by $$P^{g_k}_\gamma w=r^\frac{n+2\gamma}{2}(-\D)^\gamma u,\quad u=r^{-\frac{n-2\gamma}{2}}w.$$ For a function $w$ on $\mathbb S^{N-1}\times\HH^{k+1}$ we decompose it as $w(\theta,\zeta)=\sum_{j=0}^\infty w_j(\zeta)\vp_j(\theta),$ with $\zeta\in\HH^{k+1}$.
  
The next lemma can be found in \cite{Ao1}. 
  \begin{lem}  Let $P^j_\gamma$ be the projection of the operator $P^{g_k}_\gamma$ on the eigenspace  $\langle \vp_j\rangle $. Then    \begin{align}\label{symbol-2}\widehat{P^j_\gamma w_j}=\Theta^j_\gamma(\lambda)\hat w_j,\quad \Theta^j_\gamma(\lambda)=2^{2\gamma}\frac{|\Gamma(\frac12+\frac\gamma2+\frac12\sqrt{(\frac N2-1)^2+\lambda_j}+\frac\lambda2 i)|^2}{|\Gamma(\frac12-\frac\gamma2+\frac12\sqrt{(\frac N2-1)^2+\lambda_j}+\frac\lambda2 i)|^2},\end{align}where $\,\widehat{\cdot} \,$ denotes the Fourier-Helgason transform on $\HH^{k+1}$. \end{lem}
  
    \medskip 
    
    \noindent\emph{Proof of Proposition \ref{inj-prop-3}} As we mentioned before, we shall use Proposition \ref{inj-prop-1}. Let $\phi$ be a solution to $$\D^2\phi-pu_1^{p-1}\phi=0\quad\text{in }\R^n\setminus\R^k.$$ We set $w=r^{-\frac{n+2\gamma}{2}}\phi$. Let $w_j$ be the projection of $w$ on the eigenspace $\langle \vp_j\rangle$.  Let $\hat w_j(\lambda,\omega)$ be the Fourier-Helgason transform of $w_j$, $(\lambda,\omega)\in\R\times \mathbb S^{N-1}$. As the symbol \eqref{symbol-2} coincides with the symbol \eqref{symbol-1} for every $\omega\in \mathbb S^{N-1}$, our problem is equivalent to that of Proposition \ref{inj-prop-1}. This concludes the proof. 
    \hfill $\square$
    
In a similar way, using Proposition \ref{inj-prop-2} one can prove the following Proposition: 

\begin{prop}  \label{inj-prop-4} Solutions to $$\D^2 w-\frac{ A_p}{r^4}w=0\quad\text{in }\R^n\setminus\R^k,$$  are trivial in the space $ w\in C^{4,\alpha}_{\mu,\mu}(\R^n\setminus\R^k).$\end{prop}

%  \begin{lem} The following operators are surjective: $$L:C^{4,\alpha}_{\nu,-1}(\R^N\setminus\{0\})\longrightarrow  C^{0,\alpha}_{\nu-4,-1}(\R^N\setminus\{0\})$$ and   $$\mathbb{L}:C^{4,\alpha}_{\nu,-1}(\R^n\setminus\R^k)\longrightarrow  C^{0,\alpha}_{\nu-4,-1}(\R^n\setminus\R^k).$$  \end{lem}

  \section{Injectivity of $L_{\bar\ve}$ on $C^{4,\alpha}_{\mu,\mathscr{N}}(\Omega\setminus\Sigma)$} In this section we study injectivity of the linearized operator  $$L_{\bar\ve}w:=\D^2w-pu_{\bar\ve}^{p-1}w.$$
  We shall use the following notations:  % $\Omega_{\bar \ve}:=\Omega\setminus \cup _{i=i}^KB_{\ve_i}(\Sigma_i)$. For a function $f$ we shall use the notations 
  $$\Omega_{\bar \ve}:=\Omega\setminus \cup _{i=i}^KB_{\ve_i}(\Sigma_i),\quad f^+:=\max\{f,0\},\quad f^-:=\min\{f,0\}.$$
    \begin{lem} There exists $\ve_0>0$ such that if $\ve_i<\ve_0$ for every $i$, then after a suitable normalization of $u_1$, the operator $L_{\bar \ve}$ satisfies maximum principle in $\Omega_{\bar\ve}$, that is \begin{align*}  \left\{\begin{array}{ll} L_{\bar\ve}w\geq 0&\quad\text{in }\Omega_{\bar\ve} \\ w\geq0&\quad\text{on }\partial \Omega_{\bar \ve} \\ \D w\leq0&\quad\text{on }\partial \Omega_{\bar \ve}  \end{array}\right.\quad     \Longrightarrow \quad w\geq 0\text{ and }\D w \leq0\text{ in }\Omega_{\bar\ve}.\end{align*} \end{lem}
  
  \begin{proof} Let $v$ and $\tilde v$ be given by  \begin{align*}  \left\{\begin{array}{ll} - \D v=-(\D w)^-&\quad\text{in }\Omega_{\bar \ve} \\ v=w&\quad\text{on }\partial\Omega_{\bar\ve},\quad \end{array}\right.   \left\{\begin{array}{ll}  -\D \tilde v=(\D w)^+&\quad\text{in }\Omega_{\bar \ve} \\ \tilde  v=0&\quad\text{on }\partial\Omega_{\bar\ve}.\quad \end{array}\right.  \end{align*} Then $v\geq0$ and $\tilde v\geq 0$ in $\Omega_{\bar\ve}$. We shall show that $\tilde v=0$,  and hence   $v=w$. 
  
  It follows that  \begin{align*}  \left\{\begin{array}{ll} - \D (v-w)=(\D w)^+\geq 0&\quad\text{in }\Omega_{\bar \ve} \\ v-w=0&\quad\text{on }\partial\Omega_{\bar\ve},\quad \end{array}\right.   \left\{\begin{array}{ll}  -\D (\tilde v+w)=-(\D w)^-\geq0&\quad\text{in }\Omega_{\bar \ve} \\ \tilde  v+w\geq 0&\quad\text{on }\partial\Omega_{\bar\ve}.\quad \end{array}\right.  \end{align*}  Therefore, $$v-w\geq0,\quad  -w\leq\tilde v\quad \text{in }\Omega_{\bar\ve}\quad \text{and }\frac{\partial (v-w)}{\partial \nu}\leq0\text{ on }\partial\Omega_{\bar\ve},$$ where $\nu$ is the outward unit normal vector.   We compute \begin{align*} \int_{\Omega_{\bar\ve}} (v-w)\D^2 wdx&=\int_{\Omega_{\bar\ve}}\D(v-w)\D wdx+\int_{\partial\Omega_{\bar \ve}}\left((v-w)\frac{\partial \D w}{\partial\nu}-\D w\frac{\partial (v-w)}{\partial\nu}\right) d\sigma \\ &=-\int_{\Omega_{\bar \ve}}[(\D w)^+]^2dx -\int_{\partial\Omega_{\bar \ve}} \D w\frac{\partial (v-w)}{\partial\nu}d\sigma \\ & \leq  -\int_{\Omega_{\bar \ve}}[(\D w)^+]^2dx.\end{align*}  Thus \begin{align}\label{31} \int_{\Omega_{\bar \ve}}[(\D w)^+]^2dx &\leq -\int_{\Omega_{\bar \ve}}(v-w)\D^2 wdx \notag\\ &\leq p\int_{\Omega_{\bar \ve}} u_{\bar\ve}^{p-1}(-w)(v-w) dx\notag\\ &\leq  p\int_{\Omega_{\bar\ve}}u^{p-1}_{\bar\ve}\tilde v(v-w)dx\notag \\&\leq \delta \int_{\Omega_{\bar\ve}}\sum_{i=1}^K\frac{1}{|x-x_i|^4} \tilde v (v-w)dx\notag\\ & \leq \delta \sum_{i=1}^K\left(\int_{\Omega_{\bar\ve}} \frac{\tilde v(x)^2}{|x-x_i|^4} dx \right)^\frac12\left(\int_{\Omega_{\bar\ve}} \frac{(v(x)-w(x))^2}{|x-x_i|^4} dx \right)^\frac12,\end{align} where $\delta>0$ can be chosen arbitrarily small by normalizing $u_1$ so that $pr^4u_1^{p-1}(r)\leq \delta$ for $r\geq 1$. Since $\tilde v=0$ on $\partial\Omega_{\bar\ve}$ (same arguments for $v-w$),  integrating by parts we obtain \begin{align*}  \int_{\Omega_{\bar\ve}} \frac{\tilde v(x)^2}{|x-x_i|^4} dx & %=  \int_{\Omega_{\bar\ve}} \frac{\tilde v(x)^2}{|x-x_i|^4} dx \\ &
  =\frac{-1}{2(n-4)}  \int_{\Omega_{\bar\ve}}\tilde v(x)^2 \D \frac{1}{|x-x_i|^2} dx\\ &=\frac{-1}{n-4}  \int_{\Omega_{\bar\ve}}  \frac{\tilde v\D\tilde v+|\nabla\tilde v|^2}{|x-x_i|^2} dx\\ &\leq \frac{-1}{n-4}  \int_{\Omega_{\bar\ve}}  \frac{\tilde v\D\tilde v}{|x-x_i|^2} dx\\ &\leq\frac{1}{n-4} \left( \int_{\Omega_{\bar\ve}}  \frac{\tilde v(x)^2}{|x-x_i|^4} dx \right)^\frac12    \left( \int_{\Omega_{\bar\ve}}  (\D\tilde v (x))^2 dx\right)^\frac12,  \end{align*} which gives  $$ \int_{\Omega_{\bar\ve}} \frac{\tilde v(x)^2}{|x-x_i|^4} dx\leq \frac{1}{(n-4)^2}\int_{\Omega_{\bar\ve}} (\D \tilde v(x))^2dx= \frac{1}{(n-4)^2}\int_{\Omega_{\bar\ve}} [(\D w(x))^+]^2dx.$$ Going back to \eqref{31} $$ \int_{\Omega_{\bar \ve}}[(\D w)^+]^2dx\leq \delta K \frac{1}{(n-4)^2} \int_{\Omega_{\bar \ve}}[(\D w)^+]^2dx,$$ and hence $(\D w)^+=0$.
  
  We conclude the lemma. 
  
    \end{proof}

  \begin{rem} $L_{\bar\ve}$ satisfies maximum principle on  $\cup_{i=1}^K B_\sigma(\Sigma_i) \backslash B_{\ve_i}(\Sigma_i)$ for $0<\ve_i<\ve_0$.    \end{rem}
 
 \begin{lem}\label{7.2}Fix $\ve_0>0$ such that $L_{\bar \ve}$ satisfies the maximum principle on $\Omega_{\bar\ve}$. Let $4-N<\gamma<0$ be fixed. Let  $w_\bve$ be a solution to $L_\bve w_\bve=f_\bve$ on $\Omega_\bve$ for some $f_\bve\in C^{0,\alpha}_{\gamma-4}(\Omega_\bve)$, and $0<\ve_i\leq\ve_0$. Assume that $w_\bve=\D w_\bve=0$ on $\partial\Omega$. Then there exists $C>0$ such that \begin{align}\label{wve-est}\|w_\bve\|_{4,\alpha,\gamma}\leq C\left( \|f_\bve\|_{0,\alpha,\gamma-4}+\sum_{i=1}^K\left(\ve_i^{-\gamma}\|w_\bve\|_{C^0(\partial B_{\ve_i}(\Sigma_i))}+\ve_i^{2-\gamma}\|\D w_\bve\|_{C^0(\partial B_{\ve_i}(\Sigma_i))} \right)\right) .\end{align}  \end{lem} 
 \begin{proof} Let $\sigma>0$ be as in Section \ref{section-higher} so that $\bar u_\ve$ is supported in $\cup_{i=1}^KB_{\sigma}(\Sigma_i)$. We fix a smooth positive function $\phi$ on $\Omega\setminus \cup \Sigma_i$ such that $\phi(x)=d(x,\Sigma_i)^\gamma$ in each $B_\sigma(\Sigma_i)$.   
For simplicity we assume that $\Sigma_i$ is a point $x_i$. Then $\phi(x)=|x-x_i|^\gamma$ on $B_\sigma(\Sigma_i)$. We compute $$\D^2 \vp(x)=c_{N,\gamma}|x-x_i|^{\gamma-4},\quad c_{N,\gamma}:=\gamma(\gamma-2)(N^2+\gamma^2+2N\gamma-6N-6\gamma+8)>0,$$ and $$\D \phi(x)=\tilde c_{N,\gamma}|x-x_i|^{\gamma-2},\quad \tilde c_{N,\gamma}:=\gamma(N+2-\gamma)<0.$$  This shows that for a suitable choice of $u_1$, we have for some $\delta>0$ $$L_\bve \phi(x)=\D^2\phi-p\bar u_\bve^{p-1}\phi\geq \delta |x-x_i|^{\gamma-4}\quad\text{on }{\bf \Omega}:=\cup_{i=1}^K B_{\sigma}(\Sigma_i)\setminus B_{\ve_i}(\Sigma_i).$$ Therefore, we can choose $c_{1,\bve}\approx \|f_\bve\|_{0,\alpha,\gamma-4}$ so that $$  L_\bve(w_\bve+c_{1,\bve}\phi)\geq0  \quad\text{on } \bf\Omega.$$ We can also choose \begin{align*}c_{2,\bve}\approx  & \sum_{i=1}^K\left(\ve_i^{-\gamma}\|w_\bve\|_{C^0(\partial B_{\ve_i}(\Sigma_i))}+\ve_i^{2-\gamma}\|\D w_\bve\|_{C^0(\partial B_{\ve_i}(\Sigma_i))} \right) \\  &\quad +\sum_{i=1}^K\left(\|w_\bve\|_{C^0(\partial B_{\sigma}(\Sigma_i))}+\|\D w_\bve\|_{C^0(\partial B_{\sigma}(\Sigma_i))} \right)\\&=: c_{3,\bve}+c_{4,\bve}, \end{align*} so that $$ w_\bve+(c_{1,\bve}+c_{2,\bve})\phi\geq 0\quad \text{and }  \D w_\bve+(c_{1,\bve}+c_{2,\bve})\D \phi\leq 0 \quad\text{on } \bf\Omega. $$ Then by Maximum principle we have that (to get the other inequality use $-\phi$) $$|w_\bve|\leq  (c_{1,\bve}+c_{2,\bve})\phi \quad\text{and }|\D w_\bve|\leq -(c_{1,\bve}+c_{2,\bve})\D\phi\quad\text{in }\bf\Omega.$$ Since, $\D^2w_\bve=f_\bve$ in $\Omega_\bve\setminus\bf\Omega$, we get that $$|w_\bve (x) |+|\D w_\bve (x)|\lesssim (c_{1,\bve}+c_{2,\bve})\quad x\in \Omega_\bve\setminus\bf\Omega.$$ We claim that \begin{align}c_{4,\bve}\lesssim c_{3,\bve}+\|f_\bve\|_{0,\alpha,\gamma-4}.\end{align} We assume by contradiction that the above claim is false. Then there exists a family of solutions $w_\ell=w_{\bve_\ell}$ to $L_{\bve_\ell}w_\ell=f_\ell$ with $0<\ve_{i,\ell}<\ve_0$, $f_\ell \in C^{0,\alpha}_{\gamma-4}(\Omega_{\bve_\ell})$, $w_\ell=\D w_\ell=0$ on $\partial\Omega$ such that \begin{align}\label{assump-33}c_{4,\bve_\ell}=1\quad\text{and } c_{3,\bve_\ell}+\|f_\ell\|\to0. \end{align} Then, up to a subsequence, $\Omega_{\bve_\ell}\to \tilde\Omega$, where $\tilde\Omega_\bve=\Omega\setminus\cup_{i=1}^KB_{\ve_i}(\Sigma_i)$ for some $0\leq \ve_i\leq\ve_0$. Here  $B_{\ve_i}(\Sigma_i)=\Sigma_i$ if $\ve_i=0$ for some $i$. 

From the estimates on $w_\ell$ we see that $w_\ell\to w$  in $\bar \Omega\setminus\cup_{i=1}^KB_{\ve_i}(\Sigma_i)$. Moreover, $w$ satisfies  $$L_{\bar \ve}w=0\quad\text{in }\tilde\Omega_\bve,$$ where $L_\bve=\D^2-p\bar u_\bve^{p-1}$, with the understanding that if $\ve_i=0$ for some $i$ then  $\bar u_\bve=0$  on $B_\sigma(\Sigma_i)$. Notice that $w$ satisfies $$w=\D w=0\quad\text{on }\partial\Omega\cup_{\ve_i\neq 0}\partial B_\ve(\Sigma_i).$$ If $\ve_i=0$ for some $i$, then $w$ is bi-harmonic in $B_\sigma(\Sigma)\setminus \Sigma_i $, and as $w(x)=O(d(x,\Sigma_i)^{\gamma})$ with $4-N<\gamma<0$, we see that the singularity on $\Sigma_i$ is removable. Thus, we can use maximum principle to conclude that $w=0$ in $\tilde\Omega_\bve$. This contradicts the first condition in \eqref{assump-33}.

In this way we have that there exists $C>0$ independent of $\bve$, but  depending only on the right hand side  of \eqref{wve-est}  such that  $$|w_\bve|\leq C\phi\quad\text{and }|\D w_\bve|\leq C(1+|\D\phi| )\quad\text{in }\Omega_\bve. $$
The desired estimate follows from Schauder theory.

 \end{proof}

 \begin{lem}  \label{7.3} Let $(w_\ell)\subset C^{4,\alpha}_\mu(B_\sigma(\Sigma_i))$ be a sequence of solutions to $L_1w_\ell=0$ in $B_\sigma(\Sigma_i)$, for some fixed $i\in\{1,2,\dots,K\}.$  If $|w_\ell|+|\D w_\ell|\leq C$ on $B_\sigma(\Sigma_i)\setminus B_\frac\sigma2(\Sigma_i)$ then $\|w_\ell\|_{C^{4,\alpha}_\mu(\Sigma_i)}$ is uniformly bounded. \end{lem} 
 \begin{proof} It suffices to show that $$S_\ell:=\sup (r^{-\mu}|w_\ell|+r^{2-\mu}|\D w_\ell| )\leq C.$$ We assume by contradiction that the above supremum is not uniformly bounded.  Let $x_\ell=(r_\ell,\theta_\ell,y_\ell)\in B_\sigma(\Sigma_i)$  be such that  $$ S_\ell \approx r_\ell^{-\mu}|w_\ell (x_\ell)|+r_\ell^{2-\mu}|\D w_\ell (x_\ell)| .$$  We claim that $r_\ell\to 0$. On the contrary, if $r_\ell\to r_\infty\neq0$, then setting  $\bar w_\ell=\frac{w_\ell}{S_\ell}$ we see that $\bar w_\ell\to \bar w_\infty$, where $$ L_1\bar w_\infty=0\quad \text{in }B_\sigma(\Sigma_i),\quad \bar w_\infty\equiv 0\text{ in }B_\sigma(\Sigma_i)\setminus B_\frac\sigma2(\Sigma_i).$$ Therefore, $\bar w_\infty\equiv 0$ in $B_\sigma(\Sigma_i)$, which contradicts to $$r_\infty^{-\mu}|\bar w_\infty (x_\infty)|+r_\infty^{2-\mu}|\D w_\infty(x_\infty)|  \approx 1 .$$  
If $\Sigma_i=\{x_i\}$, we set    $$\tilde w_\ell(r,\theta)=\frac{r_\ell^{-\mu}w_\ell(rr_\ell,\theta)}{S_\ell},\quad 0<r<\frac{\sigma}{r_\ell}.$$  Then $$r^{-\mu}|\tilde w_\ell(r,\theta)|+r^{2-\mu}|\D \tilde w_\ell(r,\theta)|\lesssim 1 \approx \sup_{\partial B_1}(|\tilde w_\ell|+|\D\tilde w_\ell|),$$ and $\tilde w_\ell$ satisfies  $$L_{r_\ell}\tilde w_\ell=0\quad\text{in }B_\frac{\sigma}{r_\ell}.$$ If $\Sigma_i$ is higher dimensional, then $y_\ell\to y_\infty$, and we choose Fermi coordinates around $y_\infty$ so that $y_\infty=0$, and the coordinates are defined for $|y|<\tau$ for some $\tau>0$. Then we set $$\tilde w_\ell(r,\theta,y):=\frac{r_\ell^{-\mu}w_\ell(rr_\ell,\theta, r_\ell(y{ +\tilde{y}_\ell))}}{S_\ell},\quad \tilde{y_\ell}:=\frac{y_\ell}{r_\ell}\,\quad 0<r<\frac{\sigma}{r_\ell},\, |y|<\frac{\tau}{2r_\ell}.$$ In this case  $\tilde w_\ell$ satisfies the equation $L_{r_\ell}\tilde w_\ell =o(1)$. 

In both cases we have that $\tilde w_\ell\to \tilde w_\infty\not\equiv 0$,  $\tilde w_\infty$ satisfies  $r^{-\mu} |\tilde w_\infty|\leq C$. For the point singularity case, the limit function satisfies  $$\D^2 \tilde w_\infty =\frac{A_p}{r^4}\tilde w_\infty\quad\text{in }\R^n\setminus\{0\},$$ where $A_p$ is given by \eqref{Ap}.  By proposition \ref{inj-prop-2} we get that $\tilde w_\infty\equiv0$, a contradiction.  

 For the higher dimensional case $$\D^2 \tilde w_\infty =\frac{A_p}{r^4}\tilde w_\infty\quad\text{in }\R^n\setminus \R^k,$$    and we get a contradiction by Proposition \ref{inj-prop-4}. %We\color{blue} $w_\infty$ does not depend on $y$?
 \end{proof}

\begin{lem} \label{inj-omega}There exists $\ve_0>0$ sufficiently small such that if each $\ve_i<\ve_0$ then $$L_\bve:C^{4,\alpha}_{\mu,\mathscr{N}} (\Omega\setminus\Sigma)\to C^{0,\alpha}_{\mu-4}(\Omega\setminus\Sigma)$$ is injective.   \end{lem}
\begin{proof} We assume by contradiction that for every  $\ve_0^\ell:=\frac1\ell$ there exists $\bve^\ell$  with each $\ve_i^\ell<\ve_0^\ell$ such that $L_{\bve^\ell}$ is not injective. Let $w_\ell \in C^{4,\alpha}_{\mu,\mathscr{N}} (\Omega\setminus\Sigma)$ be a  non-trivial solution to $L_{\bve^\ell}w_\ell=0$. 
 We normalize $w_\ell$ so that $$\max_{\partial\Omega_{\bve^\ell}} \left( \rho(x)^{-\mu}|w_\ell(x)|+ \rho(x)^{2-\mu}|\D w_\ell(x)|\right)=1.$$ Then by Lemma \ref{7.2} we get  that \begin{align}\label{est-35}\sup_{\Omega_{\bve^\ell}} \left( \rho(x)^{-\mu}|w_\ell(x)|+ \rho(x)^{2-\mu}|\D w_\ell(x)|\right)\leq C.\end{align} 
 
  First consider the case when $\Sigma$ is a set of finite points.
 Assume that the above maximum is achieved  on $\partial B_{\ve_j^\ell}(x_j)$ for some $j$, and upto a translation, assume that $x_j=0$. We set $$\bar w_\ell(x)=(\ve_j^\ell)^{-\mu} w_\ell(\ve_j^\ell x),\quad |x|<\frac{\sigma}{\ve_j^\ell}.$$ Then $L_1\tilde w_\ell=0$ on $B_{R_\ell}$ for some $R_\ell\to\infty$, and  $r^{-\mu}|\tilde w_\ell| +r^{2-\mu}|\D \tilde w_\ell|\leq C$ for  $1\leq r\leq \frac{\sigma}{\ve_j^\ell}$. Therefore, by Lemma \ref{7.3},   $$r^{-\mu}|\tilde w_\ell| +r^{2-\mu}|\D \tilde w_\ell|\leq C\quad\text{ for  } r\leq \frac{\sigma}{\ve_j^\ell}.$$ Hence $\tilde w_\ell\to \tilde w_\infty$, where $\tilde w_\infty $ satisfies $$L_1\tilde w_\infty=0\quad\text{in }\R^n\setminus\{0\},\quad r^{-\mu}|\tilde w_\infty| +r^{2-\mu}|\D \tilde w_\infty |\leq C .$$ Hence, by Proposition \ref{inj-prop-1} we have $\tilde w_\infty\equiv 0 $, a contradiction to $\max_{\partial B_1}|\tilde w_\infty|+|\D \tilde w_\infty|=1 $. 

Next we consider the case of higher dimensional singularity. Let $x_\ell=(r_\ell,\theta_\ell,y_\ell) $  be a point around $\Sigma_j$ for some $j$  such that $$r_\ell^{-\mu}|w_\ell(x_\ell)|+ r_\ell^{2-\mu}|\D w_\ell(x_\ell)| \approx \sup_{\Omega} \rho(x)^{-\mu}|w_\ell|+ \rho(x)^{2-\mu}|\D w_\ell| =:S_\ell. $$  We can also assume that  $r_\ell\leq \ve_j^\ell$, thanks to \eqref{est-35}. We shall take $r_\ell=\ve_j^\ell$ if they are of the same order so that   either $r_\ell=o(\ve_j^\ell)$ or $r_\ell=\ve_j^\ell$.   We choose Fermi coordinates around $y_\infty$ (limit of $y_\ell$) so that $y_\infty=0$, and the coordinates are defined for $|y|<\tau$ for some $\tau>0$. We set $$\tilde w_\ell(r,\theta,y):=\frac{r_\ell^{-\mu}w_\ell(rr_\ell,\theta, r_\ell(y +\tilde{y}_\ell))}{S_\ell},\quad \tilde{y_\ell}:=\frac{y_\ell}{r_\ell}\,\quad 0<r<\frac{\sigma}{r_\ell},\, |y|<\frac{\tau}{2r_\ell}.$$  

%{\color{blue}For $r_\ell=\ve_j^\ell$, by Lemma \ref{7.2} and \eqref{est-35}  we get that \begin{align*}   \end{align*} }

As before one gets a contradiction, thanks to Propositions \ref{inj-prop-3} and  \ref{inj-prop-4}. 
 \end{proof}

  \section{Uniform surjectivity  of $L_{\bar\ve}$ on $C^{4,\alpha}_{\mu,\mathscr{N}}(\Omega\setminus\Sigma)$} 
  
  Let $\rho$ be a  smooth function on  $\Omega\setminus\Sigma$ with positive lower bound  such that $\rho(\cdot )=dist(\cdot,\Sigma_i)$ in $ B_\sigma(\Sigma_i)$, for every $i$.  The weighted  space $L^2_{\delta}(\Omega\setminus \Sigma)$ is defined by  % {\color{blue} This definition is similar to \cite[page 60]{Ao1}. In \cite[page 7]{Mazzeo-Pacard96}, it is defined without $\rho^{-4}$. Duality pair is defined  same way in  both papers.  }   
  
  $$L^2_{\delta}(\Omega\setminus \Sigma):=\left\{ w\in L^2_{loc}(\Omega\setminus\Sigma): \int_{\Omega} \rho^{-4-2\delta}|w|^2 dx<\infty \right\}.$$ Let $L^2_{-\delta}(\Omega\setminus\Sigma)$ be the dual of $L^2_{\delta}(\Omega\setminus \Sigma)$ with respect to the pairing $$L^2_{\delta}(\Omega\setminus \Sigma)\times L^2_{-\delta}(\Omega\setminus \Sigma)\,\ni\, (w_1,w_2)\longrightarrow \int_{\Omega}w_1w_2 \rho^{-4}dx.$$ We note that  the following embedding is continuous  $$C^{k,\alpha}_{\gamma}(\Omega\setminus\Sigma) \hookrightarrow L^2_{\delta}(\Omega\setminus\Sigma)\quad\text{for }\delta<\gamma+\frac{N-4}{2}.$$ 
  The domain $D(L_\bve)$ of the operator $L_\bve$ is the set of functions $w\in L^2_\delta$ (for simplicity we drop the domain $\Omega\setminus \Sigma$ from the notation)  such that  $L_\bve w=h\in L^2_{\delta-4}$ in    the sense of distributions. One can show that  the following elliptic estimate holds:  $$\sum_{\ell=1}^4 \|\nabla^\ell w\|_{L^2_{\delta-\ell}(\Omega_\sigma)}\leq C(\|h\|_{L^2_{\delta-4}} +\|w\|_{L^2_\delta }),$$  where $\Omega_\sigma:=\cup_{i=1}^K(B_\sigma(\Sigma_i) \setminus \Sigma_i)$.  Moreover, for  $\delta-\frac{N-4}{2}\not\in\{ \Re{\gamma_j^{\pm\pm}}:j=0,1,\dots\}$, the operator  $L_\bve: L^2_\delta\to L^2_{\delta-4}$ is densely defined,  it has closed graph, and  in fact, it  is Fredholm (see \cite{mazzeoEdge}). 
  
 %   For this choice of $\delta$, the following crucial estimate holds: $$\|w\|_{L^2_{-\delta}}\leq C(\|L_\bve w\|_{L^2_{-\delta-4}}+\|w\|_{L^2(\mathcal{K})}),$$ for some  fixed compact set   $\mathcal{K}\subset \bar\Omega\setminus \cup_{i=1}^K\Sigma_i$ independent of $w$. }

    We shall fix $\delta>0 $ slightly bigger than $\mu+\frac{N-4}{2}$, % simply in the region $$\mu+\frac{N-4}{2}<\delta<\frac{N-4}{2},$$
      where $\mu$ is fixed according to \eqref{mu-nu}.  The adjoint of the operator \begin{align}\label{op}L_\bve :L^2_{-\delta}\to L^2_{-\delta-4}\end{align} is given by \begin{align} \label{op-adj}L^2_{\delta+4}\to L^2_{\delta},\quad  w\mapsto \rho^4L_\bve(w\rho^{-4}).\end{align} Then the adjoint operator    \eqref{op-adj} is injective, and $L_\bve$ in \eqref{op} is surjective.  Using the isomorphism  $$\rho^{2\delta}:L^2_{\tilde \delta}\to L^{2}_{2\delta+\tilde\delta},\quad w\mapsto \rho^{2\delta}w,$$ we identify the adjoint operator as $$L_\bve^*: L^2_{-\delta+4}\to L^2_{-\delta},\quad w\mapsto \rho^{4-2\delta}L_\bve(w\rho^{2\delta-4}) .$$  Now we consider  the composition $$\LL=L_\bve \circ L_\bve^* :L_{-\delta+4}^2\to L_{-\delta-4}^2, \quad w\mapsto L_\bve[\rho^{4-2\delta}L_\bve(w\rho^{2\delta-4})].$$ Then $\LL$ is an isomorphism, and hence there exists a two sided inverse  $$\GG_\bve:L^2_{-\delta-4}\to L^2_{-\delta+4}.$$ Consequently, the right inverse of $L_\bve$ is given by $G_\bve:=L_\bve^*\GG_\bve$.  It follows from \cite{Mazzeo-Pacard96} that $$G_\bve: C^{0,\alpha}_{\nu-4} (\Omega\setminus\Sigma)\to C^{4,\alpha}_{\nu,\mathcal{N}}(\Omega\setminus\Sigma)$$
is bounded.

\begin{lem} Let $\ve_0>0$ be as in Lemma \ref{inj-omega}. Then  the system  $L_\bve w_1=0$, $w_1=L_\bve^* w_2$ with $w_1\in C^{4,\alpha}_{\nu,\mathcal{N}}(\Omega\setminus\Sigma)$ and $w_2\in C^{8,\alpha}_{\nu+4,\mathcal{N}}(\Omega\setminus\Sigma)$ has only   trivial solution.  \end{lem}
\begin{proof} We set $w=\rho^{2\delta-4}w_2$. Then $L_\bve[\rho^{4-2\delta} L_\bve w]=0$. Multiplying the equation by $w$ and then integrating by parts we get $$0=\int_{\Omega}\rho^{4-2\delta}|L_\bve w|^2dx.$$ Since $\nu+2\delta>\mu$, we have  $w\in C^{4,\alpha}_{\nu+2\delta}(\Omega\setminus\Sigma)\subset C^{4,\alpha}_{\mu}(\Omega\setminus \Sigma)$. Then by  Lemma \ref{inj-omega} we get that $w=0$, equivalently  $w_1=w_2=0$.  \end{proof}

\begin{lem}  There exists $\ve_0>0 $ small such that  if $0<\ve_i<\ve_0$ for every $i$, then the  sequence of solutions $(w_{1,\bve^\ell})\subset C^{4,\alpha}_{\nu,\mathcal{N}}(\Omega\setminus\Sigma) \cap L^*_{\bve^\ell}[C^{8,\alpha}_{\nu+4,\mathcal{N}}(\Omega\setminus\Sigma)]$  to  $L_{\bve^\ell}w_{1,\bve^\ell}=f_{\bve^\ell}$ is uniformly bounded in $C^{4,\alpha}_{\nu}(\Omega\setminus\Sigma)$, provided $(f_{\bve^\ell})$ is uniformly bounded in $C^{0,\alpha}_{\nu-4}(\Omega\setminus\Sigma)$. \end{lem}
\begin{proof} Assume by contradiction that the lemma is false. Then there exists a sequence  of $K$-tuples $(\bve^\ell)$ with $\bve_i^\ell\to 0$ for each $i=1,\dots,K$, and  $w_{1,\bve^\ell }\in  C^{4,\alpha}_{\nu,\mathcal{N}}(\Omega\setminus\Sigma)\cap L_{\bve^\ell}[C^{8,\alpha}_{\nu+4,\mathcal{N}}(\Omega\setminus\Sigma)]$  with   $L_{\bve^\ell}w_{1,\bve^\ell}=f_{\bve^\ell}$ such that $\|f_{\bve^\ell}\|_{C^{0,\alpha}_{\nu-4,\mathcal{N}}(\Omega\setminus\Sigma)}\leq C$.  By Lemma \ref{7.2} $$\|w_{1,\bve^\ell}\|_{C^{4,\alpha}_{\nu}(\Omega_{\bve^\ell})}\leq C+C \max_{\cup \partial B_{\ve_i^\ell}(\Sigma_i)}  \left(  \rho^{-\nu}|w_{1,\bve^\ell}|+\rho^{2-\nu}|\D w_{1,\bve^\ell}|\right)=:C+CS_{\bve^\ell}.$$ 

First we consider the case when $\Sigma$ is a set of finitely many points. We distinguish the following two cases. \\

\noindent\textbf{Case 1} $S_{\bve^\ell}\leq C$.

In this case we proceed as in the proof of Lemma \ref{7.3}. Let $x_\ell=(r_\ell,\theta_\ell)$ be such that $$\sup_{\cup  B_{\ve_i^\ell}(\Sigma_i)}  \left(  \rho^{-\nu}|w_{1,\bve^\ell}|+\rho^{2-\nu}|\D w_{1,\bve^\ell}|\right)\approx   \rho^{-\nu}(x_\ell)|w_{1,\bve^\ell}(x_\ell)|+\rho^{2-\nu}(x_\ell)|\D w_{1,\bve^\ell}(x_\ell)|=:S_\ell\to\infty.$$ Up to a subsequence, $x_\ell\in B_{\ve_i^\ell}(\Sigma_i)$ for some fixed $i$. Then necessarily $r_\ell =o(\ve_i^\ell)$. Setting $$\tilde w_{1,\bve^\ell}(r,\theta):=\frac{r_\ell^{-\nu}w_{1,\bve^\ell}(rr_\ell,\theta)}{S_\ell}$$ one would get that $ \tilde w_{1,\bve^\ell}\to \tilde w_1\not\equiv 0$ where $$\tilde L_1\tilde w_1=0\quad\text{in }\R^n\setminus\{0\}, \quad r^{-\nu}|\tilde w_1|\leq C,\quad \tilde L_1:=\D^2-\frac{A_p}{r^4}, $$ where $A_p$ is as in \eqref{Ap}. Since $\nu$ does not coincide with indicial roots of $\tilde L_1$, as in the proof of Proposition \ref{inj-prop-2} one obtains that  $\tilde w_1\equiv0$,   a contradiction. \\

\noindent\textbf{Case 2} $S_{\bve^\ell}\to\infty$.

In this case  we first divide the function $\tilde w_{1,\bve^\ell}$ by $S_{\bve^\ell}$. Then consider the scaling with respect to $\ve_i^\ell$ instead of $r_\ell$ (see Lemma \ref{inj-omega}) and proceeding as before we would get that  $w_{1,\bve^\ell}\to \tilde w_1\not\equiv 0$ where $$L_1\tilde w_1=0\quad\text{in }\R^n\setminus\{0\}, \quad r^{-\nu}|\tilde w_1|\leq C. $$ Since $\tilde w_1$ decays at infinity, its decay rate is determined by the indicial roots of $L_1$ (which are exactly the same as $\D^2$) at infinity. In fact, $\tilde w_1$ would be bounded by $r^{4-N}$ at infinity, see e.g.  the proof of Proposition \ref{inj-prop-1}.  

 Since $\tilde w_{1,\bve^\ell}\in L^*_{\bve^\ell}[C^{8,\alpha}_{\nu+4,\mathcal{N}}(\Omega\setminus\Sigma)]$,  we have $w_{1,\bve^\ell}=\rho^{4-2\delta}L_{\bve^\ell} w_{2,\bve^\ell}$ for some $w_{2,\bve^\ell}\in C^{8,\alpha}_{\nu+2\delta,\mathcal{N}}(\Omega\setminus\Sigma)$. As $2\delta+\nu>\mu$, applying Lemmas \ref{7.2} and \ref{7.3}, we can show that the scaled functions $$\tilde w_{2,\bve^\ell}(r,\theta):=\frac{\ve_{i,\ell}^{-\nu-2\delta}w_{2,\bve^\ell}(r\ve_{i,\ell},\theta)}{S_\ell},$$ converges to a limit function $\tilde w_2$, where $$ L_1\tilde w_2= r^{2\delta-4}\tilde w_1\quad\text{in }\R^n\setminus\{0\}, \quad r^{-\nu-2\delta}|\tilde w_2|\leq C.$$ Thus, $ L_1[r^{4-2\delta} L_1 \tilde w_2]=0$. We multiply this equation by $\tilde w_2$ and integrate it on $\R^N$. Then  an integration by parts leads   $ L_1 \tilde w_2=0$ (this is justified because of the decay of $\tilde w_1$ at infinity, provided we choose $\delta>0$ sufficiently close to $\mu+\frac{N-4}{2}$). Again,  as $2\delta+\nu>\mu$,  by Proposition \ref{inj-prop-1} we have   $\tilde w_2=\tilde w_1=0$, a contradiction.

When $\Sigma$ is of positive dimension,  we need to do the  scaling  as in Lemmas \ref{inj-omega} and \ref{7.3}. We now prove that the limit is independent of the variable $y$. The argument is based on the theory of edge operators and their parametrices\footnote{We are very grateful to Rafe Mazzeo for explaining us the argument, already mentioned in \cite{Mazzeo-Pacard96}}. 

As in the previous section, we set $$\tilde w_{1,\bve^\ell}(r,\theta,y):=\frac{r_\ell^{-\nu}w_{1,\bve^\ell}(rr_\ell,\theta, r_\ell(y +\tilde{y}_\ell))}{S_\ell},\quad \tilde{y_\ell}:=\frac{y_\ell}{r_\ell}\,\quad 0<r<\frac{\sigma}{r_\ell},\, |y|<\frac{\tau}{2r_\ell}.$$  

We now proceed as before to show, because of the normalization, that $\tilde w_{1,\bve^\ell}\to \tilde w_1\not\equiv 0$ and 

$$\tilde L_1\tilde w_1=0\quad\text{in }\R^n\setminus \R^k, \quad r^{-\nu}|\tilde w_1|\leq C,\quad \tilde L_1:=\D^2-\frac{A_p}{r^4}. $$ 

{\bf Claim: The function $\tilde w_1$ does not depend on $y \in \R^k$}. By standard theory in edge calculus (see \cite{mazzeoEdge}), each operator $L_{\bve^\ell}$ has a left parametrix $G_{\bve^\ell}$ since the solutions are normalized. In other words, there exists a compact (in the sense of pseudo-differential operators) $R_{\bve^\ell}$ such that 
$$
G_{\bve^\ell} L_{\bve^\ell}=Id+R_{\bve^\ell}
$$
along every sequence $\bve^\ell$. Furthermore, since $R_{\bve^\ell}$ is compact, it maps polyhomogeneous functions into functions with fast decay. Applying the previous identity to  $\tilde w_{1,\bve^\ell}$, one sees right away that $\tilde w_{1,\bve^\ell}$ is itself polyhomogeneous. Consider now any derivative $\partial_y^\alpha w_{1,\bve^\ell}$, denoted for simplicity $w^{(\alpha)}_{1,\bve^\ell}$. By appropriately normalizing the latter function and using the fact that the compact operator $R_{\bve^\ell}$ is itself polyhomogeneous in $y$, one gets passing to the limit in the previous equation that the limiting function has to be in kernel of $\tilde L_1$ with faster decay, hence it is identically zero. 
Hence the function $\tilde w_1$ is independent of $y$. Therefore,  we are back to the case of a point singularity. This proves the lemma.

\end{proof}

  \section{Fixed point arguments}
  
  To prove existence of solution to \eqref{eq-v} we use a fixed point argument on the space $C^{4,\alpha}_{\nu,\mathcal{N}}(\Omega\setminus\Sigma)$. Since we need to find $v$ such that $\bar u_\bve+v$ is positive in $\Omega$, we shall solve the equation \begin{align}\label{41}\left\{ \begin{array}{ll}  \D^2 (\bar u_\bve+v)=|\bar u_\bve+v|^p &\quad\text{in }\Omega \\          \rule{0cm}{.5cm}      v =\D v=0 &\quad\text{on }\partial\Omega.   \end{array}\right. \end{align} Equivalently, $$ L_\bve v+f_\bve +Q(v)=0,\quad Q(v):=-|\bar u_\bve+v|^p+\bar u_\bve^p+p\bar u_\bve^{p-1}v,$$ where  $f_\bve=\D^2 \bar u_\bve-\bar u_\bve^p$ as before. Applying the inverse of $L_\bve$, that is $G_\bve$, we rewrite the above equation as $$v+G_\bve f_\bve +G_\bve Q(v)=0.$$ The crucial fact we shall use is that the norm of $G_\bve$ is uniformly bounded if $\ve_0$ is sufficiently small. 
  
   We note that if  $v\in C^{4,\alpha}_{\nu,\mathcal{N}}(\Omega\setminus\Sigma)$ is a weak solution to the above equation then by maximum principle we have that $\bar u_\bve +v>0$ in $\Omega$. This  is a simple consequence of the fact that $\nu>-\frac{4}{p-1}$, and therefore, $\D (\bar u_\bve+v)<0$ and $\bar u_\bve+v>0$ in a small neighborhood of $\Sigma$, thanks to the asymptotic behavior of $\bar u_\bve$, $\D \bar u_\bve$ around the origin.   
 
 We recall that the error   $f_\bve=\D^2 \bar u_\bve-\bar u_\bve^p$ satisfies  the estimate   $\|f_\bve\|_{0,\alpha,\nu-4}\leq C\ve_0^{N-\frac{4p}{p-1}}$ if $\Sigma$ is a discrete set, and  $\|f_\bve\|_{0,\alpha,\nu-4}\leq C \ve_0^q$, $q=\frac{p-5}{p-1}-\nu$ otherwise. Let us first consider the case when $\Sigma$ is a set of finitely many points.  Then, there exists $C_0>0$ such that $\|G_\bve f_\bve \|_{4,\alpha,\nu}\leq C_0 \ve_0^{N-\frac{4p}{p-1}}$. This suggests to work on the ball $$\mathscr{B}_{\ve_0,M}=\left\{  v\in C^{4,\alpha}_{\nu} :\|v\|_{4,\alpha,\nu}\leq M\ve_0^{N-\frac{4p}{p-1}}\right\},$$ for some $M>2C_0$ large.    We shall show that the map $v\mapsto G_\bve[f_\bve+Q(v)]$ is a contraction on the ball  $\mathscr{B}_{\ve_0,M}$.  To   this end we shall assume that  $\ve_i\in [a\ve_0,\ve_0]$ for every $i=1,\dots,K$ for some fixed $a\in (0,1)$.    
 
 \begin{lem}  Let $M_1>1$ be fixed. Then for $\ve_0<<1$ we have $$\|Q(v_1)-Q(v_2)\|_{0,\alpha,\nu-4}\leq\frac{1}{M_1}\|v_1-v_2\|_{4,\alpha,\nu}\quad\text{for every }v_1,\,v_2\in \mathscr{B}_{\ve_0,M}.$$ \end{lem}
 \begin{proof} We start by showing that there exists $0<\tau<\sigma$ small, independent of $\ve_0<<1$, such that \begin{align}\label{est-tau-1}|v(x)|\leq \frac{1}{10}\bar u_\bve(x)\quad \text{for every }x\in \cup_{i=1}^K B_\tau(x_i),\, v\in \mathscr{B}_{\ve_0,M}.\end{align} To prove this we recall that for any fixed $R>1$  there exists $c_1,\,c_2>1$ such that $$\frac{1}{c_1}\leq |x|^\frac{4}{p-1}u_{\ve_i}(x)\leq c_1\quad \text{for }|x|\leq R\ve_0,$$ $$\frac{1}{c_2}\leq \ve_0^{-N+\frac{4p}{p-1}}|x|^{N-4}u_{\ve_i}(x)\leq c_2\quad\text{for }R\ve_0\leq  |x|\leq\tau .$$ On the other hand, $$\ve_0^{-N+\frac{4p}{p-1}}  \rho(x) ^{-\nu} |v(x)|\leq M .$$ As $\nu>4-N$, we have \eqref{est-tau-1} for some $\tau>0$ small. 
 
 We have  \begin{align*}  Q(v_1)-Q(v_2)&=\int_0^1\frac{d}{dt}|\bar u_\bve +v_1+t(v_2-v_1) |^pdt +p\bar u_\bve^{p-1}(v_1-v_2)\\ &=p(v_2-v_1)\int_0^1\left( |\bar u_\bve +v_1+t(v_2-v_1) |^{p-1}-\bar u_\bve^{p-1}\right)dt \\  &=: p(v_2-v_1)\int_0^1Q(v_1,v_2)dt. \end{align*}Next,  using that $$ (1+a)^{p-1}=1+O(|a|)\quad \text{for }  |a|\leq\frac{1}{2},$$ %and that $$\bar u_\bve(x)\leq C\rho(x)^{-\frac{4}{p-1}}\quad\text{in }\Omega, $$
  we estimate for $ x\in \cup_{i=1}^KB_{R\ve_0}(x_i)$  \begin{align}   |Q(v_1,v_2)|(x) & \leq C\bar u_\bve(x)^{p-2}(|v_1|(x)+|v_2|(x))    \notag \\ &\leq CMR^{\frac{4}{p-1}+\nu}\ve_0^{N-4+\nu} \rho^{-4}(x) , \notag  \end{align} and for  $ x\in \cup_{i=1}^K(B_\tau(x_i)\setminus B_{R\ve_0}(x_i))$  \begin{align}  \notag |Q(v_1,v_2) |(x) & \leq C_{\tau,R}M \max\{\ve_0^{N-\frac{4p}{p-1}},\, \ve_0^{(N-\frac{4p}{p-1})(p-1)}\} \rho^{-4}(x).   \end{align} For $ x\in \Omega\setminus \cup_{i=1}^KB_{\tau}(x_i)$  \begin{align}|Q(v_1,v_2)|(x) &\leq C(\bar u_\bve^{p-1}+|v_1|^{p-1}+|v_2|^{p-1} )(x)  \notag \\ &\leq  C_{\tau,M} \ve_0^{(N-4)p-N}\rho^{-4}(x)  , \notag \end{align} where in the last inequality we have used that, in this region,  $$\bar u_\bve (x)+|v_1(x)|+|v_2(x)|\leq C_{\tau,M}\ve_0^{N-\frac{4p}{p-1}}.$$ Combining these estimates we  get   for $\ve_0<<1$ $$\|Q(v_1)-Q(v_2)\|_{0,0,\nu-4}\leq c_{\ve_0}\|v_1-v_2\|_{4,\alpha,\nu},$$ where $c_{\ve_0}\to0$ as $\ve_0\to0$.  
  
 In order  to estimate the weighted  H\"older norm of $Q(v_1)-Q(v_2)$ we note that the function $|\bar u_\bve+v|^{p-1}$ is only $C^{0,p-1}$ for   $1<p<2$, which in turn implies that $Q(v_1,v_2)$ is only $C^{0,p-1}$. % At the same time, the norm $\|v_1-v_2\|$ on the right hand side corresponds to ``one-derivative'' and  the H\"older norm of order $\alpha$ corresponds to ``$\alpha$-derivative'', that is, we are estimating total derivative of order $1+\alpha$. 
 This suggests that we need to take the H\"older exponent $\alpha\leq p-1$. % Here we only outline the proof for $1<p\leq 2$, as the case $p>2$ is  easier. 

For $0<s<\sigma$ we write \begin{align*} &s^{4-\nu+\alpha}\sup_{x,x'\in N_s\setminus N_\frac s2}\frac{|[Q(v_1)-Q(v_2)](x)-[Q(v_1)-Q(v_2)](x')|} {|x-x'|^{\alpha}} \\ &\leq 4 \|Q(v_1)-Q(v_2)\|_{0,0,\nu-4} \\&\quad +s^{4-\nu+\alpha}\sup_{x,x'\in N_s\setminus N_\frac s2,\, |x-x'|\leq\frac s4 }\frac{|[Q(v_1)-Q(v_2)](x)-[Q(v_1)-Q(v_2)](x')|} {|x-x'|^{\alpha}} . \end{align*} Notice that for $x,x'\in N_s\setminus N_\frac s2$ with $|x-x'|\leq \frac s4$, the line segment $[x,y]$ joining $x$ and $y$ lies in $N_{2s}\setminus N_\frac s4$. The desired estimate follows on the region  $\cup_{i=1}^KB_\tau(x_i)$  by estimating $Q(v_1,v_2)(x)-Q(v_1,v_2)(x')$ using the following  gradient bound (we are using that $|\bar u_\bve+v|^{p-1}$ is $C^1$ in this region) \begin{align*}\nabla Q(v_1,v_2)&=(p-1)\left[ (\bar u_\bve+v_1+t(v_2-v_1)^{p-2}-\bar u_\bve^{p-2}) \right]\nabla \bar u_\bve  \\ &\quad +  (p-1) (\bar u_\bve+v_1+t(v_2-v_1)^{p-2}\nabla[v_1+t(v_2-v_1)] \\ &=O(1)\bar u_\bve^{p-3}(|v_1|+|v_2|)|\nabla u_\bve|+O(1) \bar u_\bve^{p-2}(|\nabla v_1|+|\nabla v_2|).\end{align*}  In fact,  gradient bounds can  also be used for the region $\Omega\setminus\cup_{i=1}^K B_\tau(x_i)$ if $p\geq 2$. For $1<p\leq 2$,   one can use the following inequalty  $$||\phi|^{p-1}(x)-|\phi|^{p-1}(x')|\leq |\phi(x)-\phi(x')|^{p-1}\leq \|\nabla \phi\|^{p-1}_{C^0([x,x'])}|x-x'|^{p-1},$$ with $\phi=\bar u_\bve$ and $\phi=\bar u_\bve+v_1+t(v_2-v_1)$.   

We conclude the lemma. 

 \end{proof}
 
 From the above lemma we see that for a suitable choice of $M$ and $M_1$, the map   $v\mapsto G_\bve[f_\bve+Q(v)]$ is a contraction on the ball  $\mathscr{B}_{\ve_0,M}$   onto itself, for $\ve_0<<1$. Hence, we get a solution to \eqref{41} as desired. \\

 When $\Sigma$ is not discrete, one shows  in a similar way that the map  $v\mapsto G_\bve[f_\bve+Q(v)]$ is a contraction on the ball
 $$\mathscr{B}_{\ve_0,M}=\left\{  v\in C^{4,\alpha}_{\nu} :\|v\|_{4,\alpha,\nu}\leq M\ve_0^q\right\},$$ for some suitable $M>>1$. Here $q=\frac{p-5}{p-1}-\nu$, and the parameter $\nu$ satisfies $-\frac{4}{p-1}<\nu<\min\{\frac{p-5}{p-1},\, \Re{(\gamma_0^{--})}\}$.

 \section{Proofs of Theorems \ref{main-theorem} and \ref{main2-theorem}}
 
 This section is devoted to the completion of the proofs of the theorems stated in the introduction. In the previous section, we constructed {\sl for a fixed $\bar \varepsilon$} a solution of \eqref{eq-domain}. 
 \vspace{0.5cm} 
 
 \noindent{{\sl Proof of Theorem \ref{main-theorem}: }} As noticed already in \cite{Mazzeo-Pacard96}, the modifications are very minor. Recall the equation
 \begin{align} P_{g_0}u=u^\frac{n+4}{n-4}\quad\text{in }M\setminus \Sigma,\notag \end{align}
where $g_0$ is a fixed metric. Using Fermi coordinates and the rescaled Delaunay-type solutions shows that, since the linearized operator is the bilaplacian with {\sl lower order terms}, those terms disappear in the rescaling/blow-up and one can prove in an exactly parallel way that the linearization is uniformly surjective provided $\varepsilon$ is small enough. The geometric assumptions in the theorem ensures then that the constructed solution in the fixed point, is positive. 

\vspace{0.5cm}  \noindent{{\sl Proof of Theorem \ref{main2-theorem}: }} The statement follows from a combination of the solution constructed in the previous section and the application of the implicit function theorem as described in \cite{Mazzeo-Pacard96}. To get the infinite dimensionality of the solution space, we invoke as in \cite{Mazzeo-Pacard96}, the edge calculus in \cite{mazzeoEdge}.

 \section{Appendix: Singular radial solutions in $\R^N\setminus\{0\}$}
 
 In this appendix, we collect several results related to the ODE analysis of Delaunay-type solutions for our problem (see \cite{Gazzola,GWZ}). For sake of completeness, we provide the proofs. Furthermore, since we need rather fine properties of these solutions, we also straighten some of the arguments in the above-mentioned papers. 
  
 \begin{lem} Let $u$ be a  radial solution to \eqref{singu11} with $\frac{N}{N-4}<p<\frac{N+4}{N-4}$ as given by Theorem \ref{exists1}. Then \begin{align}\notag  r^4u^{p-1} (r)\leq \frac{p+1}{2}k(p,N).\end{align} \end{lem}
 \begin{proof} Set $$\tilde u(y)=|x|^{N-4}u(x),\quad x=\frac{y}{|y|^2}.$$  Then $\tilde u$ satisfies \begin{align}\label{eq-transformed}\D^2\tilde u=|y|^{\alpha}\tilde u^p\quad\text{in }\R^N,\quad \alpha:=(N-4)p-(N+4)\in (-4,0),  \end{align}  and $\tilde u$ does not have any singularity at the origin. Now we set $$\bar u(t)=r^\frac{4+\alpha}{p-1}\tilde u(r)=r^{-\frac{4}{p-1}}u(\frac1r),\quad t=\log r.$$ One checks that $$\bar u(t)\to0, \quad \bar u'(t)\to0,\quad \bar u''(t)\to0,\quad \bar u'''(t)\to0\quad \text{as }t\to-\infty. $$ Moreover, \begin{align*} \bar u''''(t)+K_3\bar u'''(t)+K_2\bar u''  +K_1\bar u'(t)+K_0\bar u(t)=\bar u^p(t), \end{align*} where  (see e.g. \cite{GWZ,Gazzola} ) \begin{align*} K_0&:= \frac{4+\alpha}{(p-1)^4}\left[2(N-2)(N-4)(p-1)^3+(4+\alpha)(N^2-10N+20)(p-1)^2   \right.  \\ &\qquad\left. -2(4+\alpha)^2(N-4)(p-1)+(4+\alpha)^3 \right] \\ &=k(p,N)  \\ K_1&:=-\frac{2}{(p-1)^3}\left[ (N-2)(N-4)(p-1)^3+(4+\alpha)(N^2-10N+20)(p-1)^2 \right.\\   &\left. \qquad   -3(\alpha^2+8\alpha+16)(N-4)(p-1)+2\alpha(\alpha^2+12\alpha+48)+128 \right] \\ &= -\frac{2}{(p-1)^3}\left[ (6N-N^2-8)p^3+(22N-N^2-56)p^2+(5N^2-14N-56)p-3N^2-8-14N  \right] ,\\  K_2&:=\frac{1}{(p-1)^2}\left[ (N^2-10N+20)(p-1)^2-6(4+\alpha)(N-4)(p-1)+6\alpha(\alpha+8)+96  \right], \\ K_3&:=\frac{2}{p-1}\left[ (N-4)(p-1)-2(4+\alpha) \right]=\frac{2}{p-1}\left[ N+4-p(N-4)\right].
 \end{align*} 
 It follows that $K_3>0$ for $\frac{N}{N-4}<p<\frac{N+4}{N-4}$, and   $K_1$ vanishes at the following points $$p_1:=\frac{N+4}{N-4},\quad  p_2^\pm:=\frac{6-N\pm 2\sqrt{N^2-4N+8}}{N-2}.$$  We also have that $p_2^-<0<p_2^+<\frac{N}{N-4}$.  In particular, as $K_1(\infty)>0$, we have that  $K_1<0$ for $\frac{N}{N-4}<p<\frac{N+4}{N-4}$. 
 
  Let us now define the energy 
  $$E(t):=\frac{1}{p+1}\bar u^{p+1}(t) -\frac{K_0}{2}\bar u(t)^2-\frac{K_2}{2}|\bar u'(t)|^2+\frac12|\bar u''(t)|^2.$$ If $\bar u'(t_1)=0$ for some $t_1\in\R$, then following the  proof of \cite[Lemma 6]{Gazzola} we  get $$E(t_1)-E(-\infty)=K_1\int_{-\infty}^{t_1}|\bar u'(t)|dt-K_3\int_{-\infty}^{t_1}|\bar u''(t)|dt\leq0.$$ Thus, $\bar u'(t_1)=0$ implies that $$\bar u^{p-1}(t_1)\leq \frac{p+1}{2}K_0.$$  The proof follows from this, and the asymptotic behavior of  $u$ at the origin.   \end{proof}

The next lemma provides uniqueness of solutions to \eqref{eq-transformed}. We start with the following lemma:

\begin{lem} Let $u$ be a non-negative bounded radial solution to \eqref{eq-transformed} on $B_1\setminus\{0\}$. Then  $u$ is H\"older continuous for every $\alpha\in (-4,0)$, and it is $C^2$ for   $\alpha\in (-2,0)$. Moreover,  \begin{align}\label{Delta-u-1} \lim_{r\to0^+} r^{-\alpha-2}\D u(r)  =c_{N,\alpha}u(0)^p ,\quad \alpha\in (-4,-2)\end{align} and \begin{align}\label{Delta-u-2} \lim_{r\to0^+} \frac{\D u(r)}{\log r}  =c_{N,\alpha}u(0)^p ,\quad  \alpha=2,\end{align} for some constant (independent of $u$)  $c_{N,\alpha}<0$.   \end{lem}
\begin{proof}  We set $$v(x):=\frac{1}{\gamma_N}\int_{B_1}\frac{1}{|x-y|^{N-4}}|y|^\alpha u^p(y)dy,\quad h(x):=u-v(x),$$ where $\frac{1}{\gamma_N}\frac{1}{|x|^{N-4}}$ is a fundamental solution of $\D^2$ in $\R^N$. Since $u$ is bounded, one easily gets that $v$ is H\"older continuous for  $\alpha\in (-4,0)$, and differentiating under the integral sign, $v\in C^2$ for $\alpha\in (-2,0)$. Thus, $h$ is a bounded biharmonic function  on $B_1\setminus\{0\}$. Therefore, the singularity at zero is removable, and  $h$ is smooth in $B_1$. This completes the proof of regularity of $u$. 

Now we prove \eqref{Delta-u-1}. We fix $0<\delta<1$ such that $\alpha\delta-\alpha-2>0$.  Using that $u$ is continuous, we estimate for $r=|x|\neq 0$ \begin{align*}\D v(x)&=c_Nu(0)^p \int_{|y|<r^\delta} \frac{|y|^\alpha}{|x-y|^{N-2}} (1+o(1))dy+O(1) \int_{1>|y|>r^\delta} \frac{|y|^\alpha}{|x-y|^{N-2}}dy \\ &= c_Nu(0)^p \int_{|y|<r^\delta} \frac{|y|^\alpha}{|x-y|^{N-2}} (1+o(1))dy+O(1)r^{\delta\alpha}, \end{align*} where $o(1)\to0$ uniformly in $y\in B_{r^\delta}$ as $r\to0$. Using a change of variable $y\mapsto |x|y$ we obtain \begin{align*}  \int_{|y|<r^\delta} \frac{|y|^\alpha}{|x-y|^{N-2}}dy&=r^{2+\alpha} \int_{|y|<r^{\delta-1}} \frac{|y|^\alpha}{|\frac{x}r-y|^{N-2}}dy= r^{2+\alpha} \int_{\R^N} \frac{|y|^\alpha}{|\frac{x}r-y|^{N-2}}dy +o(1)r^{2+\alpha},\end{align*} where the last integral is finite as $-4<\alpha<-2$, and  $o(1)\to0$ as $r\to\infty$.  
Combining these estimates, and as $h$ is smooth,   we get \eqref{Delta-u-1}.

To prove \eqref{Delta-u-2} we fix $0<\ve<<1<<R<\infty$. As before we would get $$\D v(x)=c_Nu(0)^p \int_{|y|<\ve} \frac{|y|^{-2}}{|x-y|^{N-2}} (1+o_\ve(1))dy+O_\ve(1),$$ and after a change of variable   \begin{align*} \int_{|y|<\ve} \frac{|y|^{-2}}{|x-y|^{N-2}} dy&=\int_{|y|<\frac\ve r} \frac{|y|^{-2}}{|\frac xr-y|^{N-2}}dy=\int_{R<|y|<\frac\ve r} \frac{|y|^{-2}}{|\frac xr-y|^{N-2}} dy+ O_R(1). \end{align*} Since  $$\frac{1}{|\frac xr-y|^{N-2}}=\frac{1}{|y|^{N-2}}(1+o_R(1))\quad\text{for }|y|\geq R>>1,$$ we have    \begin{align*} \int_{|y|<\ve} \frac{|y|^{-2}}{|x-y|^{N-2}} dy =|S^{N-1}|(1+o_R(1)) \log\frac1r+O_\ve(1)+O_R(1) .\end{align*}
Combining these estimates and first taking $r\to0^+$, and then taking $\ve\to0^+, \, R\to\infty$ we obtain \eqref{Delta-u-2}. 

\end{proof}

Next we prove uniqueness of  radial solutions to \eqref{eq-transformed}. We shall use the following identity: \begin{align} \label{int-identity}w(r_2)-w(r_1)=\int_{r_1}^{r_2}\frac{1}{|S^{N-1}|t^{N-1}}\int_{B_t}\D w(x)dxdt,\quad w \text{ is radial}. \end{align}

\begin{lem} Let $u_1,u_2$ be two non-negative bounded radial solutions to \eqref{eq-transformed} on $\R^N\setminus\{0\}$ with $\alpha\in (-4,0)$.  If $u_1(0)=u_2(0)$ then $u_1=u_2$ on $\R^N$.\end{lem}
\begin{proof} 
Let us first assume that \begin{align}\label{assump}\lim_{|x|\to0^+}\D \bar u(x)=0,\quad \bar u:=u_1-u_2 .\end{align} Then using \eqref{int-identity} we obtain \begin{align} \label{baru}\bar u(x)=o(1)|x|^2\quad\text{as }|x|\to0.\end{align}  By \eqref{int-identity}-\eqref{assump} \begin{align*}\D \bar u(r)&=\int_0^r\frac{1}{|S^{N-1}|t^{N-1}}\int_{B_t}|x|^\alpha(u_1^p(x)-u_2^p(x))dxdt \\ &=o(1)|\bar u|_r \int_0^r \frac{1}{t^{N-1}}\int_{B_t}|x |^{2+\alpha}\\&=o(1)|\bar u|_rr^{4+\alpha},\end{align*} where we have set  $|\bar u|_r:=\sup_{0<t<r}t^{-2}|\bar u(t)|$.  This leads to \begin{align*}  \bar u(r)=o(1)|\bar u|_r\int_{0}^r\frac{1}{t^{N-1}}\int_{B_t}|x|^{4+\alpha}dxdt =o(1)|\bar u|_r r^{6+\alpha},\end{align*}  which gives $$r^{-2}|\bar u(r)|\leq\frac12 \sup_{0<t<r}t^{-2}|\bar u(t)|\quad\text{for every }0<r\leq r_0,$$ for some $r_0>0$ sufficiently small. From this and  \eqref{baru} we get   that  $\bar u\equiv 0$ in a small neighborhood of the origin, and consequently we have $\bar u\equiv 0$ in $\R^N$. 

It remains to prove \eqref{assump}, and we do that in few steps. \\

\noindent\textbf{Step 1} Assume that $\bar u(x)=O(1)|x|^\gamma$ for some $\gamma\geq 0$. Then setting  $\tilde\gamma:=\alpha+\gamma+2$   we have \begin{align*}\D \bar u(x)=O(1)\left\{ \begin{array}{ll}   |x|^{\tilde\gamma } &\quad\text{if }\tilde\gamma<0 \\ \log |x| &\quad\text{if }\tilde\gamma=0 \\ 1 &\quad\text{if } \tilde\gamma>0,\end{array}\right.  \quad  \bar u(x)=O(1)\left\{ \begin{array}{ll}   |x|^{\tilde\gamma+2} &\quad\text{if } \tilde\gamma<0 \\ |x|^2 \log |x| &\quad\text{if } \tilde\gamma=0 \\|x|^2 &\quad\text{if } \tilde\gamma>0.\end{array}\right. \end{align*}

We set $\bar v:=v_1-v_2$, $\bar h:=h_1-h_2$ where   $$v_i(x):=\frac{1}{\gamma_N}\int_{B_1}\frac{1}{|x-y|^{N-4}}|y|^\alpha u_i^p(y)dy,\quad h_i:=u_i-v_i,\,i=1,2.$$ Then using that $|u_1^p(x)-u_2^p(x)|\leq C|\bar u(x)|\leq C|x|^\gamma$ \begin{align*} \D\bar v(x)&=c_n\int_{B_1}\frac{|y|^\alpha}{|x-y|^{N-2}}(u_1^p(y)-u_2^p(y)) dy \\ &=O(1)\int_{B_1}\frac{|y|^{\alpha+\gamma}}{|x-y|^{N-2}}dy\\&=O(1)\left\{ \begin{array}{ll}   |x|^{\alpha+\gamma+2} &\quad\text{if }\alpha+\gamma+2<0 \\ \log |x| &\quad\text{if }\alpha+\gamma+2=0 \\ 1 &\quad\text{if }\alpha+\gamma+2>0,\end{array}\right. \end{align*} thanks to Lemma \ref{lem9.4}.   First part of Step 1 follows as $\bar h$ is smooth in $B_1$.  The second part follows immediately by the first part and the identity \eqref{int-identity}. \\

\noindent\textbf{Step 2} The function $\bar u$ is $C^2$. 

Since $\bar u(x)=O(1)$, we can  use Step 1 with $\gamma=0$, and deduce that $\bar u(x)=O(|x|^{4+\alpha})$ (or the other growths at $0$).  In fact,  we can repeat this process finitely many times to eventually get that $\bar u(x)=O(|x|^2)$. Then, as $\alpha>-4$,  from the integral representation of $\bar v$ it is easy to see that $\bar v$ is $C^2$, and consequently $\bar u$ is $C^2$.  \\

\noindent\textbf{Step 3} \eqref{assump} holds. 

Since $\bar u$ is $C^2$,  $a:=\lim_{|x|\to0}\D \bar u(x)$ exists. If $a>0$  then  a repeated use of \eqref{int-identity} yields that  $\D\bar u\geq a$ on $\R^N$. In particular, $\bar u(x)\geq 2Na |x|^2$ on $\R^N$, a contradiction as $\bar u$ is bounded. The case $a<0$ is similar.
\end{proof}

Proof of the following lemma is straight forward. 
\begin{lem} \label{lem9.4}For $q_1,q_2\in (0,N)$ there exists $c=c(N,q_1,q_2)>0$ such that for any $R>0$ we have $$\lim_{|x|\to0^+}|x|^{q_1+q_2-N}\int_{B_R}\frac{dy}{|x-y|^{q_1}|y|^{q_2}}=c\quad\text{ if }q_1+q_2>N, $$ and  $$\lim_{|x|\to0^+} -(\log |x|)^{-1}\int_{B_R}\frac{dy}{|x-y|^{q_1}|y|^{q_2}}=c\quad\text{ if }q_1+q_2=N, $$ \end{lem}

\begin{thm} \label{thm-4.6}There exists a positive radial solution $u\in C^0(\R^N)\cap C^4(\R^N\setminus\{0\})$ to \eqref{eq-transformed} such that $u$ is monotone decreasing and $u$ vanishes at infinity. In fact, $u(r)\leq C r^{-\frac{4+\alpha}{p-1}}$ at infinity.   \end{thm} 

To prove the theorem we consider the auxiliary equation %(this or with Dirichlet boundary conditions?) 
\begin{align} \label{aux}\left\{\begin{array}{ll} \D^2 u=\lambda |x|^\alpha(1+u)^p &\quad\text{in }B_1\\ u={\frac{\partial u}{\partial \nu}}=0 &\quad\text{on }\partial B_1.  \end{array}\right. \end{align} Next we prove existence of a positive radial solution to \eqref{aux} for some $\lambda>0$. This will be done by Schauder fixed point theorem on the space  $$X:= C^0_{rad}(\bar B_1),\quad \|u\|:=\|u\|_{C^0(\bar B_1)}.$$ We define $T:X\to X$, $u\mapsto\bar u$ where we have set \begin{align}\label{bar-u} \bar u(x):=\int_{B_1}G(x,y)|y|^\alpha  (1+|u(y)|)^pdy,\end{align} where   $G=G(x,y)$ is the Green function for $\D^2$ on $B_1$ with Dirichlet boundary conditions. It is easy to see that $T$ is well-defined, and in fact, $T$ is   compact.  Therefore, there eixsts $0<t_0\leq 1$ and $u_0\in X$ such that $tTu_0=u_0$. Then  $u_0$ is  positive, monotone decreasing, and it satisfies \eqref{aux} with $\lambda=t_0$. 

As $u_0$ is a super soulution to \eqref{aux} for $0<\lambda\leq t_0$, one can prove existence of positive, radially symmetric, monotone decreasing,   minimal solution $u=u_\lambda$ to \eqref{aux} for every $0<\lambda\leq t_0$. 

Next we prove uniqueness of the minimal solutions  for $\lambda>0$ small.  In order to do that let us  recall the following Pohozaev identity from \cite[Theorem 7.27]{GGS}. 

\begin{lem}Let $u$ be a solution to \begin{align}\label{26-f}\left\{ \begin{array}{ll} \D^2u=f(x,u) &\quad\text{in }\Omega\subset\R^N\\ u=\frac{\partial u}{\partial\nu}=0&\quad\text{on }\partial\Omega. \end{array}\right. \end{align} Then setting $F(x,t)=\int_0^tf(x,s)ds$ we have $$\int_{\Omega}\left[   F(x,u)+\frac1Nx\cdot F_x(x,u)-\frac{N-4}{2N}(\D u)^2  \right]dx=\frac{1}{2N}\int_{\partial \Omega}(\D u)^2(x\cdot\nu)d\sigma.$$ \end{lem} 

Since $\int_{\Omega}uf(x,u)dx=\int_{\Omega}|\D u|^2dx$, and $x\cdot\nu=1$ on $\partial\Omega$ for $\Omega=B_1$,   the above Pohozaev identity leads to  \begin{align}\label{27}  \int_{B_1}\left[   F(x,u)+\frac1Nx\cdot F_x(x,u)-  \sigma uf(x,u)\right]dx  \geq \left(  \frac{N-4}{2N}-\sigma\right)\int_{B_1}(\D u)^2dx. \end{align} We also need the following Hardy-Sobolev inequality:  \begin{align}\label{H-S}   \left(\int_{B_1}\frac{|u|^{\frac{2(N-\beta)}{N-4}}}{|x|^\beta}dx\right)^\frac{N-4}{N-\beta}\leq c_0\int_{B_1}|\D u|^2dx\quad\text{for }u\in H^2_0(B_1),\end{align} where $B_1\subset\R^N,\,N\geq 5$ and $0<\beta<4$. This can  be derived  from the  Sobolev inequality $\|u\|_{L^{2^*}}\leq C\|\D u\|_{L^2}$ ($2^*:=\frac{2N}{N-4}$), Hardy inequality $\|\frac{u}{|x|^2}\|_{L^2}\leq C\|\D u\|_{L^2}$ and H\"older inequality.  

Using \eqref{27} and \eqref{H-S} one can prove the following lemma, see e.g. \cite[Proposition 2.2]{Ao1}. 
\begin{lem} \label{unique-minimal} There exists $\lambda_0>0$ such that for every $\lambda\in(0,\lambda_0]$ the minimal solution $u=u_\lambda$ to \eqref{aux} is the unique solution on the space $C^0(\bar B_1)$.  \end{lem}

%It follows that $u_0$ is a super solution of \eqref{aux} for every $0<\lambda<t_0$, and in particular one obtains a branch of minimal solutions $(\lambda,u_\lambda)_{0\leq \lambda\leq t_0}$ to \eqref{aux}.
From  \cite[Theorem 6.2]{Rab} we know that the closure of the set of solutions $\{(\lambda, u)\}\subset \R\times X$ to \eqref{aux}  is unbounded in $(0,\infty)\times X$.     Therefore, there exists a  unbounded sequence   $(\lambda_k,u_{\lambda_k})\in(0,\infty)\times X$ of solutions to \eqref{aux}.  Then necessarily $u_{\lambda_k}(0)=\max u_{\lambda_k}\to\infty$, and by  Lemma \ref{unique-minimal},   $\lambda_k\not\to0$.  We set $$v_k(x):=\frac{u_{\lambda_k}(r_kx)}{u_{\lambda_k}(0)},\quad r_k^{4+\alpha}\lambda_k u_{\lambda_k}(0)^{p-1}:=1.$$  Then $r_k\to 0$, and $v_k$ satisfies $$\D^2v_k=|x|^\alpha (\frac{1}{u_{\lambda_k}(0)}+v_k)^p\quad\text{in } B_\frac{1}{r_k},\quad 0\leq v_k\leq 1,\quad v_k(0)=1,\quad  \D v_k\leq 0. $$ By elliptic estimates, up to a subsequence, $v_k\to v$ locally uniformly in $\R^N$, where $v$ is a non-trivial bounded positive radial solution to $$\D^2 v=|x|^\alpha v^p\quad\text{in }\R^N.$$  Now to prove the decay estimate of $v$ at infinity, we use that $v$ is monotone decreasing, and  $\D v<0$ on $\R^N$. For $r>0$, by \eqref{int-identity}, we get \begin{align*} \D v(2r) &=\D v(r)+ \int_r^{2r}\frac{1}{|S^{N-1}|t^{N-1}}\int_{B_t}|x|^\alpha v^p(x)dxdt   \\ &\geq  \D v(r)+ \int_r^{2r}\frac{1}{|S^{N-1}|t^{N-1}}\int_{B_r}|x|^\alpha v^p(x)dxdt  \\ &\geq \D v(r)+c_1 r^{\alpha+2}v^p(r) ,\end{align*} for some constant $c_1>0$. Thus $$\D v(r)+c_1 r^{\alpha+2}v^p(r) <0,$$ which leads to  \begin{align*}  v(2r) &=v(r)+ \int_r^{2r}\frac{1}{|S^{N-1}|t^{N-1}}\int_{B_t}\D v(x)dxdt  \\ &\leq v(r)+ \int_r^{2r}\frac{1}{|S^{N-1}|t^{N-1}}\int_{B_r}\D v(x)dxdt  \\ &\leq v(r)+c_2 \D v(r) r^2\\ &\leq v(r)-c_1c_2r^{4+\alpha} v^p(r), \end{align*} for some constant $c_2>0$. 

This finishes the  proof of Theorem \ref{thm-4.6}.

\bibliographystyle{alpha} 
\bibliography{mybibfile}

\medskip

\small

\begin{center}
 Department of Mathematics \\
Johns Hopkins University \\
3400 N. Charles Street \\
Baltimore, MD 21218\\
ahyder4@jhu.edu,\,\,\,sire@math.jhu.edu
\end{center}

 \end{document}